\documentclass[11pt]{amsart}
\usepackage[margin=1in]{geometry}
\usepackage{graphicx}
\usepackage{soul}
\usepackage{amsthm, amsmath, amssymb, bm, bbm,euscript}
\usepackage{tikz}
\usetikzlibrary{positioning}
\usetikzlibrary{arrows.meta}
\usetikzlibrary{fit}
\usetikzlibrary{calc}
\usepackage{mathrsfs}
\usepackage{stmaryrd}

\usepackage{nicematrix}
\usepackage{tikz-cd}
\usepackage{framed}

\usepackage[utf8]{inputenc}
\usepackage[T1]{fontenc}
\usepackage[textsize=scriptsize,backgroundcolor=orange!5]{todonotes}

\usepackage{hyperref}
\usepackage{url}

\usepackage{mathtools}

\usepackage[noabbrev,capitalize]{cleveref}
\crefname{equation}{}{}
\usepackage{xcolor}

\usepackage{xcolor,graphics, graphicx}
\usepackage{floatrow}
\usepackage{verbatim}
\newfloatcommand{capbtabbox}{table}[][\FBwidth]

\numberwithin{equation}{section}

\newtheorem{theorem}{Theorem}[section]
\newtheorem{proposition}[theorem]{Proposition}
\newtheorem{lemma}[theorem]{Lemma}
\newtheorem{lemma/defn}[theorem]{Lemma/Definition}

\newtheorem*{question*}{Question}

\theoremstyle{definition}
\newtheorem{definition}[theorem]{Definition}

\newtheorem{example}[theorem]{Example}

\newtheorem{construction}[theorem]{Construction}
\newtheorem{notation}[theorem]{Notation}

\theoremstyle{remark}
\newtheorem{remark}[theorem]{Remark}

\newcommand{\mc}{\mathcal}
\newcommand{\mf}{\mathfrak}
\newcommand{\eu}[1]{\EuScript{#1}}

\newcommand{\res}{\mathrm{res}}

\renewcommand{\subset}{\subseteq}
\renewcommand{\supset}{\supseteq}

\newcommand{\R}{\mathbb{R}}
\newcommand{\N}{\mathbb{N}}
\newcommand{\Z}{\mathbb{Z}}
\newcommand{\C}{\mathbb{C}}

\newcommand{\Q}{\mathbb{Q}}

\DeclareMathOperator{\1}{\mathbbm{1}}

\DeclareMathOperator{\es}{\varnothing}

\newcommand{\mfp}{\mathfrak{p}}

\newcommand{\Spec}{\mathrm{Spec}}

\newcommand{\kbar}{\overline{k}}

\newcommand{\Hom}{\mathrm{Hom}}

\renewcommand{\subset}{\subseteq}

\newcommand{\msF}{\mathscr{F}}

\newcommand{\mfn}{\mathfrak{n}}

\newcommand{\Gal}{\mathrm{Gal}}
\renewcommand{\O}{\mathcal{O}}

\renewcommand{\Re}{{\rm Re}}

\renewcommand{\mod}{\ {\rm mod} \ }

\newcommand{\disc}{{\rm disc}}
\newcommand{\rad}{{\rm rad}}
\newcommand{\Sur}{{\rm Sur}}

\newcommand{\Frob}{{\rm Frob}}

\newcommand{\NN}{{\rm N}}

\DeclareMathOperator{\Res}{{\rm Res}}

\DeclareMathOperator*{\sumflat}{\sideset{}{^\flat}{\sum}}
\DeclareMathOperator*{\prodflat}{\sideset{}{^\flat}{\prod}}

\usepackage[OT2,T1]{fontenc}
\DeclareSymbolFont{cyrletters}{OT2}{wncyr}{m}{n}
\DeclareMathSymbol{\Sha}{\mathalpha}{cyrletters}{"58}

\title{Erd\H{o}s--Kac theorems for discriminants of number fields}
\author{Jack B. Miller}
\date{\today}

\allowdisplaybreaks

\begin{document}

\begin{abstract}
	The classical Erd\H{o}s--Kac theorem gives a central limit theorem for the number of prime divisors of a random integer.
	We prove an analog for the number of ramified primes in a random $G$-extension of a number field when $G$ is abelian.
	This builds on previous work of Lemke Oliver and Thorne in the cases $G = S_d$ ($2 \le d \le 5$), and provides the first examples where local ramification events at distinct primes are \textit{not} independent.

	We develop probability results that can be used ``out of the box'' to prove Erd\H{o}s--Kac theorems for sequences of ideals in a number field, subject to Tauberian hypotheses involving finite sums of Euler products.
\end{abstract}

\maketitle

\section{Introduction}
\label{sec:introduction}

A question of central importance in analytic number theory is the distribution of prime factors of a random integer $D \in\N$.
Similarly, a question of central importance in arithmetic statistics is the distribution of ramified primes of a random number field $K/\Q$.
In fact, these two questions may be combined by considering a natural invariant $D_K \in \N$, the \textit{absolute discriminant} of $K/\Q$, for which the prime factors of $D_K$ coincide with the ramified primes of $K/\Q$.

Given a random integer $n \in \N$, Erd\H{o}s and Kac \cite{ErdosKacOriginalPaper} proved a central limit theorem for the number of distinct prime factors $\omega(n)$.
Namely, they showed for every interval $I \subset \R$ that
\[
\lim_{X \to \infty} \frac{1}{X} \cdot
\#\left\{n \in \N: n\leq X,
\frac{\omega(n) - \log\log X}{\sqrt{\log\log X}} \in I
\right\} 
= \frac{1}{\sqrt{2\pi}} \int_I
e^{-t^2/2} \,dt.
\]

Given a random degree $d$ extension $K/\Q$ whose Galois closure has Galois group $S_d$, Lemke Oliver and Thorne \cite{LemkeOliverThorneErdosKac} analogously showed for all $d \in \{2,3,4,5\}$ and every interval $I \subset \R$ that
\[
\lim_{X\to\infty} \frac{1}{\#\{K/\Q : D_K \leq X\}}
\cdot \#\left\{
K/\Q : D_K \leq X, \frac{\omega(D_K) - \log\log X}{\sqrt{\log\log X}} \in I
\right\} =
\frac{1}{\sqrt{2\pi}} \int_I
e^{-t^2/2} \,dt,
\]
where $K/\Q$ varies over degree $d$ $S_d$-extensions.

In this paper we provide a framework for proving Erd\H{o}s--Kac-type results for discriminants of $G$-extensions of number fields where $G$ is a finite group.

\subsection{Abelian extensions ordered by discriminant}

We begin by providing a simple-to-state result that illustrates new features of our work.

Given an extension of number fields $L/K$ of finite degree, let $\disc(L/K) \subset \O_K$ denote the discriminant ideal, $|\disc(L/K)| \in \N$ denote the absolute discriminant, and define
\[
\omega(L/K) := \#\{\mfp \subset \O_K \text{ prime}: \mfp \mid \disc(L/K)\}.
\]

\begin{theorem}
	\label{thm:Erdos Kac abelian disc}
	Let $G$ be a nontrivial finite abelian group, and $k$ a number field.
	Let $p$ be the smallest prime dividing the order of $G$, $r$ denote the $p$-rank of $G$, and define
	\[
	b(k,G) := \frac{p^r - 1}{[k(\zeta_p):k]}.
	\]
	Then for every interval $I \subset \R$ one has that
	\begin{multline*}
	\lim_{X\to\infty} \frac{1}{\#\{K/k : |\disc(K/k)| \leq X\}} \cdot \#\left\{
	K/k : |\disc(K/k)| \leq X, 
	\frac{\omega(K/k)-b(k,G) \log\log X}{\sqrt{b(k,G) \log\log X}} \in I
	\right\}
	\\ = 
	\frac{1}{\sqrt{2\pi}} \int_I e^{-t^2/2} \,dt,
	\end{multline*}
	where $K/k$ varies over $G$-extensions.
\end{theorem}

\begin{remark}
	When $G = \Z/2\Z$ and $k = \Q$, the constant $b(k,G)$ in \cref{thm:Erdos Kac abelian disc} is equal to $1$.
	This recovers the result of Lemke Oliver and Thorne \cite[Thm.\ 1.1]{LemkeOliverThorneErdosKac} for quadratic number fields.
\end{remark}

\subsubsection{Why is the number of ramified primes $\sim b \log\log X$?}
\label{subsubsec:why b log log X many ramified primes?}

Let $G$ be a nontrivial finite abelian group and $b = b(\Q,G) \in \N$ as above.
For $k=\Q$, \cref{thm:Erdos Kac abelian disc} implies that the discriminant of a random $G$-number field $K/\Q$ is an integer lying in $[-X,X]$ with approximately $b \log \log X$ prime factors on average; this is $b$ times as many prime factors as one would expect for a random integer in $[-X,X]$.
We give a heuristic explanation for this result as follows.

We start by writing (up to some bounded factor)
\[
\disc(K/\Q) = 
\prod_{\mf{S}} \disc(K/\Q,\mf{S})
\]
where $\disc(K/\Q,\mf{S})$ is the part of the discriminant divisible by ramified primes $p \nmid |G|$ where $K/\Q$ has \textit{ramification type $\mf{S}$ above $p$} (\cref{def:ramification type}).

We expect by Malle's original heuristics \cite{MalleMCI,MalleMCII} that the ramification types of \textit{minimal index} \cite{MalleMCI} contain 100\% of the images of local tame inertia generators of a random $G$-extension, and there are $b(\Q,G)$ many minimal index ramification types of $G$ (\cref{rem:number of sectors is Malle b}).
Thus, to leading order we can factor the integer $\disc(K/\Q) \in [-X,X]$ as a product of $b$ ``independent'' integers $\disc(K/\Q,\mf{S})$ of roughly the same size, and then standard analytic-number-theory heuristics predict that the number of primes dividing each integer $\disc(K/\Q,\mf{S})$ is approximately $\log\log X$.
In particular, it would be interesting to generalize \cref{thm:Erdos Kac abelian disc} to a multivariate central limit theorem for the number and distribution of ramified primes sorted by ramification type.

Another heuristic of potential interest to the reader may be found in the survey article by Loughran and Santens \cite[\S7.2]{LoughranSantensSurvey}.

\subsubsection{A non-Billingsley random model}

Lemke Oliver and Thorne \cite{LemkeOliverThorneErdosKac} follow Billingsley's moment method \cite{BillingsleyEK} for proving Erd\H{o}s--Kac theorems.
Our proof also follows a probabilistic moment method to deduce an Erd\H{o}s--Kac result; however, our random model for primes dividing the discriminant is quite different.

Billingsley's method relies on the standard random model for integer divisibility, which predicts that the number of primes dividing a random integer $n\in\N$ can be modeled by a sequence of independent coin flips $B_p \in \{0,1\}$ for every prime $p$.
The probability that the $p$-th coin flip is heads ($B_p = 1$) is equal to $\frac{1}{p}$, and the random variables $B_p$ are referred to as Bernoulli random variables.

Using geometry-of-numbers results to count degree $d$ $S_d$-extensions for $2 \le d \le 5$, Lemke Oliver and Thorne \cite{LemkeOliverThorneErdosKac} argue that the number of primes dividing a random $S_d$-discriminant $\disc(K/\Q) \in \Z$ can be modeled by a sequence of independent coin flips $B_p^{S_d} \in \{0,1\}$.
This time, the probability that the $p$-th coin flip is heads is equal to $\frac{1}{p} + O_{d}(\frac{1}{p^2})$; however, the big-O term does not change the overall result that a central limit theorem holds because of independence, and the number of heads observed among primes $p\le X$ is still $\log\log X$ on average with standard deviation $\sqrt{\log\log X}$.

In the setting of \cref{thm:Erdos Kac abelian disc} where $G$ is abelian, the na\"ive Billingsley model $B_p^{G} \in \{0,1\}$ where $B_p^{G}$ are independent Bernoulli random variables fails to capture the true behavior of primes dividing the discriminant.
For example, Wood \cite[Prop.\ 4.1]{WoodLocalAbelianProbabilities} has shown for all primes $p_1, p_2$ congruent to $1 \mod 4$ that the events $\{K/\Q : p_1 \mid D_K\}$ and $\{K/\Q : p_2 \mid D_K\}$ are \textit{not} independent when $K/\Q$ varies over $\Z/4\Z$-extensions ordered by discriminant.
This prevents us from using the method employed by both Billingsley and Lemke Oliver--Thorne when handling the case of $\Z/4\Z$-number fields ordered by discriminant, and similarly for all $\Z/q^2\Z$-number fields where $q$ is prime.

The random model we develop goes beyond Billingsley-type random models where independence between primes is enforced.
We study the well-known family of \textit{latent class models} appearing in the statistics literature.
When one specializes a latent class model to the case of a finite mixture of Bernoulli random variables, we simply refer to this as a \textit{Bernoulli mixture model}.

For our purposes, a \textit{Bernoulli mixture model} along the sequence of primes is a probability space $\Omega$ and a family of Bernoulli random variables $(B_p)_{p \text{ prime}}$ on $\Omega$ that satisfy the following:
\begin{itemize}
	\item there exists a finite set of measurable events $\Lambda = \{\lambda_1,\ldots,\lambda_{\ell}\}$ such that $\Omega = \coprod_{i=1}^{\ell} \lambda_i$;
	\item for every $\lambda \in \Lambda$, the Bernoulli random variables $B_p$ are independent conditional on $\lambda$.
\end{itemize}
We refer to the data $\Lambda$ as the \textit{set of parameters} of the mixture model.

The reason that Bernoulli mixture models naturally arise in our proof of \cref{thm:Erdos Kac abelian disc} follows from the work of Wright \cite{WrightAbelianExtensions}.
Wright's work shows that, for a fixed number field $k$ and finite abelian group $G$, the Dirichlet series used to count abelian $G$-extensions decomposes as a finite sum of Euler products.
By suitably interpreting the set of Euler products as a parameter set (using the formalism of algebraic probability spaces \cite{NicaSpeicherLecturesOnFreeProbability} explained in \cref{sec:probability theory I}), this allows us to define an algebraic incarnation of a Bernoulli mixture model for ramified primes, which we call a \textit{Bernoulli process} (\cref{def:Bernoulli process for ideals}).
Verifying that a central limit theorem holds for Bernoulli processes that are \textit{well-mixed} (\cref{def:well-mixed Bernoulli process}) is more involved than Billingsley's method, and is a novel development of our paper.

\subsection{Abelian extensions ordered by arbitrary counting functions}
\label{subsec:general results}

Let $k$ be a number field with absolute Galois group $\Gamma_k$, and $G \le S_d$ a transitive permutation group of degree $d \in \Z_{\ge 2}$.
Let $\Sur(\Gamma_k,G)$ denote the set of continuous surjections $\varphi:\Gamma_k \twoheadrightarrow G$, which by Galois correspondence is in bijection with degree $d$ $G$-extensions $k_\varphi/k$ (see \cref{subsec:notation}).
Thus, for every $\varphi \in \Sur(\Gamma_k,G)$ we define the number-of-ramified-primes function
\[
\omega(\varphi) := \#\{\mfp \subset \O_k \text{ prime}: \mfp \text{ ramifies in } k_\varphi\}.
\]
Let $w : G \to \Z_{\ge 0}$ be a \textit{weight function}, i.e.\ a class function such that $w(g) = 0$ if and only if $g$ is the identity element, and $w(g^e) = w(g)$ for all $e \in (\Z/\exp(G)\Z)^\times$.
Wood \cite[\S2.1]{WoodLocalAbelianProbabilities} defines the notion of a \textit{counting function $C:\Hom(\Gamma_k,G) \to \R_{> 0}$ with weight $w$} in the case that $G$ is abelian, and our later definition generalizes this to every finite group $G$ (see \cref{def:height function}).

\begin{example}
	If ${\rm ind}:G\to\Z_{\ge0}$ is the \textit{index function} ${\rm ind}(g) := d - \#\{\text{cycles of }g\}$, then an example of a counting function with weight ${\rm ind}$ is $C(\varphi) := |\disc(k_\varphi/k)|$, which is exactly the absolute discriminant of the $G$-extension $\varphi$.
\end{example}

\begin{example}
	If $w:G\to \Z_{\ge0}$ is the function $w(g) = 1$ for all $g$ nontrivial, then an example of a counting function with weight $w$ is $C(\varphi) := |\rad\,\disc(k_\varphi/k)|$, which is exactly the product of ramified primes of the $G$-extension $\varphi$.
\end{example}

Let $k,G$ be as above, and $w$ a weight function.
We define \textit{Malle's $a$-constant} $a(G,w)$ to be the minimal positive weight attained by $w:G\to\Z_{\ge 0}$.
Let $\mc{M}(G)$ denote the set of minimal positive weight conjugacy classes of $G$, and define \textit{Malle's $b$-constant}
\[
b(k,G,w) := \#\Big( \mc{M}(G) \ / \ \Gal(k(\zeta_{\exp(G)})/k) \Big),
\]
where $\Gal(k(\zeta_{\exp(G)})/k)$ acts on $G$ (and consequently $\mc{M}(G)$) via $g \mapsto g^{{\rm cyc}(\sigma)}$, where ${\rm cyc}: \break \Gal(k(\zeta_{\exp(G)})/k) \to (\Z/\exp(G)\Z)^\times$ is the cyclotomic character.

We state a more general version of \cref{thm:Erdos Kac abelian disc} for $G$ abelian and $C:\Hom(\Gamma_k,G) \to \R_{>0}$ an arbitrary counting function.

\begin{theorem}
	\label{thm:Erdos Kac abelian counting function}
	Let $k,G,w,C,\omega$ be as at the beginning of \cref{subsec:general results}, and assume that $G$ is abelian.
	Then for every interval $I \subset \R$ one has that
	\begin{multline*}
		\lim_{X\to\infty} \frac{1}{\#\{\varphi : C(\varphi) \leq X\}} \cdot \#\left\{
		\varphi : C(\varphi) \leq X, 
		\frac{\omega(\varphi) - b(k,G,w)\log\log X}{\sqrt{b(k,G,w) \log\log X}} \in I
		\right\}
		= 
		\frac{1}{\sqrt{2\pi}} \int_I e^{-t^2/2} \,dt,
	\end{multline*}
	where $\varphi$ varies over $\Sur(\Gamma_k,G)$.
\end{theorem}

For $C(\varphi) := |\disc(k_\varphi/k)|$, \cref{thm:Erdos Kac abelian counting function} recovers \cref{thm:Erdos Kac abelian disc}.
In fact, similar to Wood's work \cite{WoodLocalAbelianProbabilities} on local probabilities of abelian extensions, as well as the works of Alberts \cite{AlbertsTwistedMalle}, Alberts--O'Dorney \cite{AlbertsODorney}, Darda--Yasuda \cite{DardaYasudaTorsorsFiniteGroupSchemes}, and Tavernier \cite{TavernierRestrictedRamificationAbelianCount}, we are able to prove \cref{thm:Erdos Kac abelian counting function} for all counting functions simultaneously, although the case of the discriminant already presents many of the key challenges: for a general abelian group $G$, the discriminant is neither \textit{fair} in the sense of Wood, nor \textit{balanced} in the sense of Loughran--Santens \cite{LoughranSantensMalleBrauer,LoughranSantensSurvey}.

Let $(k,G,w,C)$ be as at the beginning of \cref{subsec:general results}, and let $\eu{P}$ be a finite set of primes of $k$.
For all $X\ge 10$, the number of $G$-extensions of $k$ with $C$-invariant $\le X$ that are ramified over $\eu{P}$ admits an asymptotic expansion
\[
c_{0,\eu{P}} X^{1/a} (\log X)^{b-1} + c_{1,\eu{P}} X^{1/a} (\log X)^{b-2} + \dotsi + c_{b-1,\eu{P}} X^{1/a} + o(X^{1/a}).
\]
These lower-order terms present an additional difficulty beyond the lack of independence in the leading constant $c_{0,\eu{P}}$ as $\eu{P}$ varies.
The higher-order coefficients $c_{1,\eu{P}},\ldots,c_{b-1,\eu{P}}$ do not exhibit the same kind of ``conditional independence'' as $c_{0,\eu{P}}$, and as far as the author is aware these coefficients have not been studied in the context of local probabilities in number field counting problems.

A strategy we develop for handling the higher-order coefficients is to systematically cancel their contributions to the moments using an inclusion-exclusion process known in the probability literature as the \textit{method of cumulants}.

\subsection{Ideas of the proof}
\label{subsec:idea of proof}

We sketch the proof of \cref{thm:Erdos Kac abelian disc}, since it suggests how to understand the general methods of our paper.
Let $k,G$ be as in \cref{thm:Erdos Kac abelian disc}.

Let $S$ be a finite set of primes of $k$ and $\eu{P}$ a finite set of primes of $k$ avoiding $S$.
We define the formal Dirichlet series
\[
	D_{k,G}^{\eu{P}}(s) := \sum_{\substack{\varphi:\Gamma_k \to G \text{ surjective} \\ \text{ramified above }\eu{P}}} \frac{1}{|\disc(\varphi)|^s}.
\]
The Dirichlet coefficients of $D_{k,G}^{\eu{P}}(s)$ count (up to some constant factor) $G$-extensions of $k$ ordered by discriminant that are ramified above $\eu{P}$.
Let $D_{k,G}(s) := D_{k,G}^{\es}(s)$.

Since $G$ is a nontrivial abelian group, global class field theory (cf.\ \cite{WoodLocalAbelianProbabilities}) implies that $D_{k,G}(s)$ admits a meromorphic continuation to a right half-plane of $\C$, and $D_{k,G}(s)$ contains a pole of order $b = b(k,G,{\rm ind}) \in \N$ at $s_0 := 1/a(G,{\rm ind}) > 0$ with unique largest real part among all poles of $D_{k,G}(s)$ in this region.
Global class field theory also implies that $D_{k,G}(s)$ restricted to $\Re(s) \ge s_0$ can be written as a sum of meromorphic Euler products $(D_\lambda(s))_{\lambda \in \Lambda}$ for some finite indexing set $\Lambda$, the only potential poles of $(D_\lambda(s))_{\lambda \in \Lambda}$ in this region are at $s = s_0$, and the maximum order of these poles is at most $b$.
Thus, in the region $\Re(s) \ge s_0$ we have an equality of meromorphic functions
	\begin{equation}
		\label{eq:mixture for D k G es}
		D_{k,G}(s) = \sum_{\lambda \in \Lambda} D_\lambda(s).
	\end{equation}
	
	Moreover, if $S$ is chosen sufficiently large, then for all primes $\mfp \notin S$ the $\mfp$-local factor $1/(1 - f_\mfp^{\lambda}(s))$ of the Euler product $D_\lambda(s)$ is non-vanishing for all $\Re(s) \ge s_0$.
	This allows us to write for all $\eu{P}$ and $s\in\C$ in the region $\Re(s) \ge s_0$ that
	\begin{equation}
		\label{eq:mixture for D k G Sigma}
		D_{k,G}^{\eu{P}}(s) = \sum_{\lambda \in \Lambda} D_\lambda(s) \prod_{\mfp \in \eu{P}} f_\mfp^\lambda(s),
	\end{equation}
	where the $f_{\mfp}^{\lambda}(s)$ are holomorphic in the region $\Re(s) \ge s_0$ and do \textit{not} depend on $\eu{P}$.
After applying Perron's formula to \cref{eq:mixture for D k G es,eq:mixture for D k G Sigma}, one finds that
\begin{equation}
	\label{eq:mixture probabilities abelian disc}
	\lim_{X\to \infty} \frac{\#\{\varphi \in \Sur(\Gamma_k,G) : |\disc(\varphi)| \le X, \varphi \text{ ramified above }\eu{P}\}}{\#\{\varphi \in \Sur(\Gamma_k,G) : |\disc(\varphi)| \le X\}} = \sum_{\lambda \in \Lambda} w_\lambda \prod_{\mfp \in \eu{P}} f_{\mfp}^{\lambda}(s_0),
\end{equation}
where the $w_\lambda \in \C$ are constant and do \textit{not} depend on $\eu{P}$, and $\sum_{\lambda \in \Lambda} w_\lambda = 1$ (consider $\eu{P} = \es$).

The guiding principle is that the probabilities \cref{eq:mixture probabilities abelian disc} coincide with those of a well-mixed Bernoulli process on the primes of $k$ with parameter set $\Lambda$.
In other words, although the probabilities in \cref{eq:mixture probabilities abelian disc} for disjoint sets of primes $\eu{P}_1, \eu{P}_2$ are generally \textit{not} independent, they nevertheless are \textit{conditionally independent} after conditioning on a parameter $\lambda \in \Lambda$ (we view the complex numbers $(w_\lambda)_{\lambda \in \Lambda}$ as virtual probabilities, since $\sum_{\lambda \in \Lambda} w_\lambda = 1$).
From our perspective, this is the probabilistic interpretation of what it means for $D_{k,G}(s)$ to equal a finite sum of Euler products.

With this insight that the ramification probabilities are constrained by the class of well-mixed Bernoulli processes, it now becomes a purely statistical question to determine whether such Bernoulli random variables are weakly correlated enough that a central limit theorem should hold.

We let $Z \ge 10$ be a parameter, and for each $G$-extension $K/k$ we consider the $Z$-smoothed count of ramified primes
\[
\omega_{Z}(K/k) := \#\{\mfp \subset \O_k \text{ prime}: \mfp \mid \disc(K/k), \NN \mfp \le Z\}.
\]
Using ideas purely from probability theory, we show that the $\lim\limits_{Z \to \infty} \lim\limits_{X \to \infty}$ limit of the random variable $\omega_Z(K/k)$, uniformly sampled from $G$-extensions $K/k$ with $|\disc(K/k)| \leq X$, weakly converges to the standard normal distribution, with density $\frac{1}{\sqrt{2\pi}} e^{-t^2/2}\,dt$, after normalizing $\omega_Z(K/k)$ to have mean $0$ and variance $1$.

In order to arrive at the same conclusion but for the more difficult order of limits $\lim\limits_{X\to\infty} \lim\limits_{Z \to \infty}$, we require a uniform rate of convergence for \cref{eq:mixture probabilities abelian disc} as the set of ramified primes $\mc{P}$ varies, and this is where we use a uniform version of Malle's conjecture with local ramification conditions (\cref{prop:quantitative equidistribution for abelian extensions}).
This particular version of Malle's conjecture for abelian extensions follows from the works of Frei--Loughran--Newton \cite{FreiLoughranNewtonPowerSavingsAbelianExtensions} and Alberts \cite{AlbertsPowerSavingsAbelianExtensions,AlbertsAveragedInputTauberianTheorems}, and finishes our sketch of the proof of \cref{thm:Erdos Kac abelian disc}.

\subsection{A method of cumulants}

In the above proof sketch, we invoked the method of cumulants without explanation.
We now give an informal sketch of this method, developed more rigorously in the probability sections of our paper (\cref{sec:probability theory I,sec:probability theory II}).

For the purposes of our sketch, a \textit{joint cumulant} $\kappa(Y_1,\ldots,Y_n)$ of real-valued random variables $Y_1,\ldots,Y_n$ is a real number obtained from the \textit{joint moments} $(\mathbb{E}[\prod_{i \in I} Y_{i}])_{I \subset \{1,\ldots,n\}}$ via an inclusion-exclusion formula.
Joint cumulants are often easier to work with than the joint moments of a collection of random variables because they are $\R$-multilinear in their arguments, while still retaining complete information about the moments.

Let $(B_p)_{p \text{ prime}}$ be a sequence of Bernoulli random variables, not necessarily independent, but still satisfying certain technical assumptions about mixing.
The scale of our Bernoulli probabilities is such that $\mathbb{E}[B_p] \asymp \frac{1}{p}$.

Let $p_1,p_2,p_3$ be pairwise distinct primes (we use three primes only for illustration).
We say that the \textit{trivial bound} on the joint cumulant of the Bernoulli random variables along $p_1,p_2,p_3$ is an estimate of the form
\[
|\kappa(B_{p_1},B_{p_2},B_{p_3})|
\ll \frac{1}{p_1 p_2 p_3}.
\]
Note that if these random variables were independent, we would have $\kappa(B_{p_1},B_{p_2},B_{p_3}) = 0$.
Summing the trivial bound over all such triples $p_1,p_2,p_3 \le X$ yields an error of size $\asymp (\log\log X)^3$, which exceeds the \textit{CLT bound} $o((\log \log X)^{3/2})$ predicted by a central limit theorem for $\sum_{p\le X} B_p$.
Thus, our goal in the method of cumulants is to overcome the trivial bound for $\kappa(B_{p_1},B_{p_2},B_{p_3})$, as well as all higher joint cumulants of the form $\kappa(B_{p_1},\ldots,B_{p_n})$, $n\ge 3$.

In the limit as $X \to \infty$, the Bernoulli random variables $B_p$ associated to a finite sum of Euler products (in the sense of \cref{subsec:idea of proof}) become \textit{conditionally independent}, and this allows us to prove (using the $\R$-multilinearity of joint cumulants) a bound that is roughly of the form
\[
|\kappa(B_{p_1},B_{p_2},B_{p_3})| \ll 
\frac{1}{(p_1p_2p_3)^{1.01}}.
\]
This agrees with the CLT bound, since the total cumulance of all such triples $p_1,p_2,p_3 \le X$ is an error of size $O(1)$.
On the other hand, for finite values of $X$, the above conditional independence bound is roughly of the form
\begin{equation}
\label{eq:conditional independence bound not strong enough}
|\kappa(B_{p_1},B_{p_2},B_{p_3})| \ll 
\frac{1}{(p_1p_2p_3)^{1.01}} + \frac{1}{p_1 p_2 p_3} \frac{\log p_1 p_2 p_3}{\log X}.
\end{equation}
The above bound applied to all such triples $p_1,p_2,p_3 \le X$ yields an error of size $\asymp (\log \log X)^{2}$, which exceeds the CLT bound.

Removing the limit as $X\to \infty$ and obtaining a CLT bound for finite $X$ is the most technical part of our paper (\cref{subsec:bounding algebraic cumulants}).
The idea is that the factor of $\log p_1p_2p_3 / \log X$ appearing in \cref{eq:conditional independence bound not strong enough} arises from \textit{differentiation} in the Cauchy residue formula for a non-simple pole appearing in a Tauberian-style computation of the joint moments of $B_{p_1},B_{p_2},B_{p_3}$.
By using multilinearity of cumulants to ``force each $p_i$-local factor in the Euler products to be differentiated at least once,'' we obtain a bound that is roughly of the form
\[
|\kappa(B_{p_1},B_{p_2},B_{p_3})| \ll 
\frac{1}{(p_1p_2p_3)^{1.01}} + \frac{1}{p_1 p_2 p_3} \frac{\log p_1}{\log X} \frac{\log p_2}{\log X} \frac{\log p_3}{\log X}.
\]
This agrees with the CLT bound, since the total cumulance of all such triples $p_1,p_2,p_3 \le X$ is an error of size $O(1)$.

\subsection{General Erd\H{o}s--Kac phenomena}
\label{subsec:general EK conjecture}

We expect that the probabilistic methods of our paper can be used to prove Erd\H{o}s--Kac theorems in other settings.
For example, one can ask whether \cref{thm:Erdos Kac abelian disc,thm:Erdos Kac abelian counting function} generalize to arbitrary global fields $k$, finite groups $G$, and counting functions $C$, and in many cases we expect the answer to be yes.

\subsubsection{Relationship with the work of Loughran and Santens}

Loughran and Santens \cite{LoughranSantensMalleBrauer,LoughranSantensSurvey} have conjectured that for any counting function $C$ that is \textit{balanced}, meaning that the minimal positive weight conjugacy classes generate $G$ (cf.\ \cite[Def.\ 8.1]{LoughranSantensMalleBrauer}), the leading constant in Malle's conjecture is equal to a finite sum of Euler products indexed by certain Brauer classes of the classifying stack $BG$.
An explicit prediction for the leading constant in Malle's conjecture for $(k,G,C)$ with finitely many local conditions is implied by their \textit{equidistribution conjecture} \cite[Conj.\ 9.10]{LoughranSantensMalleBrauer}, and this potentially suggests (with further conjectural work required, along the lines of \cite[Conj.\ 8.8]{LoughranSantensSurvey}) that the sequence of ramified primes of $G$-extensions of $k$ ordered by $C$ (i.e.\ the sequence of discriminants) is approximated by a well-mixed Bernoulli process in the sense of \cref{def:reasonable sequence for ideals}(ii).

\subsection{Outline of the paper}
\label{subsec:outline of paper}

In \cref{sec:probability theory I}, we state a probabilistic criterion (\cref{thm:EK criterion for ideals}) for proving Erd\H{o}s--Kac theorems for strongly additive functions on sequences of ideals in a number field.
We also introduce key analytic hypotheses that are necessary input for \cref{thm:EK criterion for ideals}.
In \cref{sec:heights}, we establish a convention for height functions and the ramification type of a prime.
In \cref{sec:Erdos Kac abelian extensions}, we combine the statements of \cref{sec:probability theory I,sec:heights} and a uniform version of Malle's conjecture with local conditions (\cref{prop:quantitative equidistribution for abelian extensions}) to deduce \cref{thm:Erdos Kac abelian counting function}.
Finally, in \cref{sec:probability theory II} we develop our method of cumulants and prove our probabilistic criterion for Erd\H{o}s--Kac theorems (\cref{thm:EK criterion for ideals}).

\subsection{Notation}
\label{subsec:notation}

The set of natural numbers is $\N = \{1,2,3,\ldots\}$.

When a sum, product, or tuple is indexed by the symbol $\mfp$ without further comment, $\mfp$ ranges over the nonzero prime ideals of the ambient Dedekind domain $\O$, and $\NN\mfp$ denotes the absolute norm of $\mfp$ as an ideal, i.e.\ $[\O:\mfp]$.
If $L/K$ is a finite extension of number fields, we often write $|\disc(L/K)|$ for the absolute norm of the discriminant ideal $\disc(L/K)$ in the ring of integers $\O_K$ of $K$.

When $k$ is a number field, we write $\kbar$ for a fixed algebraic closure of $k$ and $\Gamma_k := \Gal(\kbar/k)$ for its absolute Galois group, endowed with the profinite topology.
Finite groups are always given the discrete topology.
Accordingly, whenever $G$ is a finite group, the sets $\Hom(\Gamma_k,G)$ and $\Sur(\Gamma_k,G)$ are understood to consist of continuous homomorphisms and continuous surjections respectively.

If $d\in\N$, $G \leq S_d$ is a transitive permutation group, and $\varphi \in \Sur(\Gamma_k,G)$, we write $k_\varphi/k$ for the finite extension corresponding to the open subgroup $\varphi^{-1}({\rm Stab}_G(1)) \leq \Gamma_k$, $\widetilde{k}_\varphi/k$ for its Galois closure (corresponding to the open normal subgroup $\ker(\varphi)$), and $\disc(\varphi) := \disc(k_{\varphi}/k)$.

For complex-valued functions $f$ and $g$, we write $f = O_{\star}(g)$ or $f \ll_{\star} g$ to denote that there exists an effectively computable constant $C>0$ (depending at most on $\star$) such that in a subset $\mc{U} \subset\C$ that will be clear from context, we have for all $z \in \mc{U}$ that $|f(z)| \le C |g(z)|$.

If $S$ is a set, an \textit{indexed subset} $\msF$ of $S$ means a function
\[
\iota_{\msF}: I_{\msF} \to S
\]
from some set $I_{\msF}$, called the \textit{indexing set} of $\msF$.
When $S$ is implied, we refer to $\msF$ as an \textit{indexed set}.
In practice we suppress the map $\iota_{\msF}$ and simply write $\msF$, while remembering that the same element of $S$ may occur several times with different indices.
Similarly, given $m\in\N$ and an $m$-tuple $T = (s_1,\ldots,s_m)$ of elements of $S$, we often interpret $T$ as an indexed subset of $S$ via
\[
\iota_T:\{1,\ldots,m\} \to S, \qquad 
\iota_T(j) = s_j.
\]

If $T$ is a set and $\msF$ is an indexed subset of $S$, a \textit{function} $\msF \to T$ means a function $I_{\msF} \to T$ that factors through $\iota_{\msF}: I_{\msF} \to S$, and expressions such as
\[
\#\{x \in \msF : \Phi(x)\}
\]
are understood to count indices in $I_{\msF}$ whose elements in $S$ satisfy the property $\Phi$, so multiplicities of $\iota_{\msF}$ are retained.
Similarly, the sum and product expressions
\[
\sum_{x \in \msF} f(x), \qquad 
\prod_{x \in \msF} g(x)
\]
when appropriately well-defined for functions $f,g:S \to \C$ retain the multiplicities of $\iota_\msF$.
When we wish to consider sums or products of complex-valued functions on $S$ indexed by $\msF$ \textit{without multiplicity}, we instead write
\[
\sumflat_{x\in\msF} f(x), \qquad 
\prodflat_{x \in \msF} g(x).
\]
If we write a tuple $(t_{x})_{x \in \msF}$ with $t_x \in T$, this is precisely the data of a function $\msF \to T$, and the tuple retains the multiplicities of $\msF$.
When we wish to consider a tuple of elements of $T$ indexed by $\msF$ \textit{without multiplicity}, we instead write $(t_{x})_{x \in \msF}^{\flat}$.

\subsection*{Acknowledgements}

We would like to thank Brandon Alberts, Christopher Frei, Daniel Loughran, Robert Lemke Oliver, Tim Santens, Julie Tavernier, Frank Thorne, and Melanie Matchett Wood for helpful conversations and/or comments on an earlier draft of this manuscript, Andrew Granville for a helpful conversation that inspired \cref{subsubsec:why b log log X many ramified primes?} and a sharpening of \cref{def:reasonable sequence for ideals}, and Ofir Gorodetsky for originally teaching the author about Billingsley's proof of the classical Erd\H{o}s--Kac theorem.
The author was partially supported by the James Mills Peirce Fellowship at Harvard University, and thanks the Centre de Recherches Mathématiques at the University of Montr\'eal for excellent working conditions, as well as Victoria Miller and Luc Vallée for their hospitality.

The author acknowledges the use of AI tools, including GPT-5.4 by OpenAI, for developing intuition about the scope of the results and for helping organize \cref{sec:probability theory II} into multiple lemmas.
The author is responsible for all statements, proofs, and exposition in this paper.


\section{Probability theory I: ideas and statements}
\label{sec:probability theory I}

In this section, we establish a framework for proving Erd\H{o}s--Kac theorems for sequences of ideals in a fixed number field under the assumption that a sequence of ideals is reasonable in the sense of \cref{def:reasonable sequence for ideals}.
The main output of this section is \cref{thm:EK criterion for ideals}.

\subsection{An Erd\H{o}s--Kac criterion}

Throughout this subsection, we let $k$ be a number field, and $\O_k$ the Dedekind domain corresponding to the ring of algebraic integers in $k$.
Over the natural numbers, a \textit{strongly additive function} \cite{NoteAdditiveFunctions} $\omega:\N \to \C$ is a function that satisfies
\[
\omega(n) = \sum_{\substack{p \text{ prime} \\ p \mid n}} \omega(p).
\]
In order to generalize this notion of a strongly additive function on integers to a strongly additive function on discriminant ideals $\disc(K/k) \subset \O_k$, where $K/k$ is some finite extension, we make the following definition.

\begin{definition}
	We define $\mf{N}_k$ to be the set of nonzero ideals of $\O_k$.
	A function $\omega:\mf{N}_k \to \C$ is \textit{strongly additive} if for all $\mfn \in \mf{N}_k$ one has that
	\[
	\omega(\mfn) = \sum_{\mfp \mid \mfn} \omega(\mfp).
	\]
\end{definition}

\begin{definition}
	Given a strongly additive function $\omega:\mf{N}_k \to \C$, we define its \textit{sup-norm} to be
	\[
	\Vert \omega \Vert_{\rm sup} := \sup_{\mfp} |\omega(\mfp)|.
	\]
\end{definition}

We are interested in the statistical behavior of a strongly additive function $\omega$ along a sequence of ideals in $\mf{N}_k$ ordered by a Northcott height, which we define as follows.

\begin{definition}
	\label{def:ordered sequence of ideals in Ok}
	A \textit{sequence of ideals} $\msF$ in $\mf{N}_k$ is an indexed subset of $\mf{N}_k$.
	A \textit{Northcott height} on $\msF$ is a function $H:\msF \to \R_{>0}$ such that for all $X > 0$ one has that
	\[
	\#\{\mfn \in \msF : H(\mfn) \le X\} < \infty.
	\]
	We refer to the pair $(\msF,H)$ as an \textit{ordered sequence of ideals in $\O_k$}, and for all $X > 0$ we write
	\[
	\msF_H(X) := \{\mfn \in \msF : H(\mfn) \le X\}.
	\]
\end{definition}

It is useful to think of $\mfn \in \msF_H(X)$ as being uniformly random.
If $\omega:\mf{N}_k \to \R$ is strongly additive with bounded sup-norm, and the random ideal $\mfn$ is ``reasonably well-behaved'' (see \cref{def:reasonable sequence for ideals}), then we obtain an Erd\H{o}s--Kac central limit theorem for $\omega(\mfn)$.

\begin{theorem}[Erd\H{o}s--Kac criterion over $\mf{N}_k$]
	\label{thm:EK criterion for ideals}
	Let $k$ be a number field, $\omega:\mf{N}_k \to \R$ a strongly additive function with bounded sup-norm, and $(\msF,H)$ an ordered sequence of ideals in $\O_k$.
	Let $(\msF,H)$ be a reasonable sequence of ideals in $\O_k$ (see \cref{def:reasonable sequence for ideals}).
	For all $X \ge 10$, write
	\[
	\mu_X :=
	\frac{1}{\log\log X}
	\sum_{\mfp} \omega(\mfp)
	\frac{\#\{\mfn \in \msF_H(X) : \mfp \mid \mfn\}}{
		\# \msF_H(X)
	},
	\]
	\[
	V_X :=
	\frac{1}{\log\log X}
	\sum_{\mfp} \omega(\mfp)^2\,
	\frac{\#\{\mfn \in \msF_H(X) : \mfp \mid \mfn\}}{
		\# \msF_H(X)
	}
	\left(
	1 - 
	\frac{\#\{\mfn \in \msF_H(X) : \mfp \mid \mfn\}}{
		\# \msF_H(X)
	}
	\right).
	\]
	Then the following statements hold.
	\begin{enumerate}
		\item[(i)]
	Suppose that $V_X$ converges to $V > 0$ as $X \to \infty$.
	Then for every interval $I \subset \R$ one has that
	\[
		\lim_{X\to\infty} \frac{1}{\#\msF_H(X)} \cdot
		\#\left\{\mfn \in \msF_H(X) : \frac{\omega(\mfn) - \mu_X \log\log X}{\sqrt{V \log\log X}} \in I
		\right\} = 
		\frac{1}{\sqrt{2\pi}} \int_{I} e^{-t^2/2} \,dt.
	\]

	\item[(ii)]
	If $(\msF,H)$ is reasonable with respect to a well-mixed Bernoulli process $\mc{B} = \mc{B}(\Lambda,s_0,S,(\boldsymbol{f}_\mfp(s))_{\mfp \notin S})$ (see \cref{def:Bernoulli process for ideals,def:well-mixed Bernoulli process}), then the limits $\mu = \lim\limits_{X \to \infty} \mu_X$ and $V = \lim\limits_{X\to \infty} V_X$ respectively exist if and only if the following limit formulae for $\mu$ and $V$ respectively exist:
	\[
	\mu = \sum_{\lambda \in \Lambda} w_\lambda(\infty) \lim_{X\to\infty} \frac{1}{\log\log X} \sum_{\substack{\mfp \notin S \\ \NN\mfp \le X}} \omega(\mfp) f_\mfp^\lambda(s_0),
	\]
	\[
	V = \sum_{\lambda \in \Lambda} w_\lambda(\infty) \lim_{X\to\infty} \frac{1}{\log\log X}
	\sum_{\substack{\mfp \notin S \\ \NN\mfp \le X}} \omega(\mfp)^2 f_\mfp^\lambda(s_0).
	\]
	\item[(iii)]
	Assume that $\mu_X$ and $V_X$ converge to $\mu \in \R$ and $V>0$ respectively as $X \to \infty$, as well as
	\[
	\sum_{\lambda \in \Lambda} w_{\lambda}(\infty)
	\sum_{\substack{\mfp \notin S \\ \NN\mfp \le X}} \omega(\mfp) f_\mfp^\lambda(s_0) = 
	\mu \log\log X + o(\sqrt{\log\log X}).
	\]
	Then for every interval $I \subset \R$ one has that
	\[
		\lim_{X\to\infty} \frac{1}{\#\msF_H(X)} \cdot
		\#\left\{\mfn \in \msF_H(X) : \frac{\omega(\mfn) - \mu \log\log X}{\sqrt{V \log\log X}} \in I
		\right\} = 
		\frac{1}{\sqrt{2\pi}} \int_{I} e^{-t^2/2} \,dt.
	\]
	\end{enumerate}
\end{theorem}

For example, when $k = \Q$, $\mf{N}_{\Q} \cong \N$, $\msF \subset \N$, and $H(n) = n$ is the ordinary height on positive integers, then \cref{thm:EK criterion for ideals} establishes a criterion for proving that a strongly additive function $\omega:\N\to\R$ satisfies a central limit theorem.

In the remainder of this section, we explain the definition of reasonability appearing in \cref{thm:EK criterion for ideals}, and postpone the proof of the theorem until \cref{sec:probability theory II}.

\subsection{Algebraic probability theory: motivation}

We find that the natural counting tools from analytic number theory -- namely Perron's formula and Tauberian theorems -- are naturally formalized using algebraic (rather than classical) probability spaces \cite{NicaSpeicherLecturesOnFreeProbability}.
The reason is simple: given a countable collection of arithmetic objects $\Omega$ and a function $H:\Omega \to \R_{\ge 1}$ (with $H^{-1}([1,10])$ nonempty), the \textit{expected value} of an arithmetic statistic $F:\Omega\to\C$ up to height $X \in \R_{\ge 10} \backslash H(\Omega)$, which is given by
\[
\mathbb{E}_{\Omega,H,X}[F] :=
\frac{\sum_{x\in\Omega, H(x) \le X} F(x)}{\sum_{x\in\Omega, H(x) \le X} 1},
\]
is often better understood in terms of \textit{Perron's formula}:
\[
\frac{\sum_{x\in\Omega, H(x) \le X} F(x)}{\sum_{x\in\Omega, H(x) \le X} 1} =
\lim_{T\to\infty}
\frac{\int_{2-iT}^{2+iT} \widehat{F}(s) X^s/s \,ds}{\int_{2-iT}^{2+iT} \left(\sum_{x\in \Omega} H(x)^{-s}\right) X^s/s \,ds},
\]
where $\widehat{F}(s)$ is the \textit{Mellin transform} of $F(x)$, given by
\[
\widehat{F}(s) := 
\sum_{x\in\Omega} F(x) H(x)^{-s}.
\]
Here we have assumed that both $\zeta_{\Omega,H}(s) := \sum_{x\in \Omega}H(x)^{-s}$ and $\widehat{F}(s)$ converge absolutely for all $s\in\C$ with $\Re(s) > 1$.
Moreover, if $\zeta_{\Omega,H}(s)$ and $\widehat{F}(s)$ admit meromorphic continuations to $\Re(s)\ge 1$ and have unique poles at $s=1$ in this region, then we often expect that applying a Tauberian theorem should yield
\[
\frac{\sum_{x\in\Omega, H(x) \le X} F(x)}{\sum_{x\in\Omega, H(x) \le X} 1} = 
\frac{\underset{s=1}{\Res} \Big(\widehat{F}(s) X^s/s\Big)}{\underset{s=1}{\Res}\Big(\zeta_{\Omega,H}(s) X^s/s\Big)} + 
\text{error}.
\]

From the perspective of trying to compute an arithmetic-statistics limit $c := \lim\limits_{X \to \infty} \mathbb{E}_{\msF,H,X}[F]$ (if it exists), it makes sense to restrict our analysis to the class of functions $F:\Omega \to \C$ for which
\[
\widehat{F}(s) = f(s) \, \zeta_{\Omega,H}(s)
\]
where $f(s)$ is holomorphic in the region $\Re(s) \ge 1$, in which case one expects that $f(1) = c$.
We refer to $f(s)$ as the \textit{normalized Mellin transform} of $F(x)$.
In particular, we write
\[
\frac{\underset{s=1}{\Res} \Big(\widehat{F}(s) X^s/s\Big)}{\underset{s=1}{\Res}\Big(\zeta_{\Omega,H}(s) X^s/s\Big)} = 
\frac{\underset{s=1}{\Res} \Big(f(s)\, \zeta_{\Omega,H}(s) X^s/s\Big)}{\underset{s=1}{\Res}\Big(\zeta_{\Omega,H}(s) X^s/s\Big)}
\]
and refer to this ratio as the \textit{algebraic expectation} of $f(s)$, which we denote $\widehat{\mathbb{E}}_{\zeta,X}[f]$.

We expect that the expected value of an arithmetic statistic $F:\Omega\to\C$ is closely approximated by the algebraic expectation of the normalized Mellin transform of $F$ as $X\to \infty$:
\[
\mathbb{E}_{\Omega,H,X}[F] =
\widehat{\mathbb{E}}_{\zeta,X}[f] + {\rm error}.
\]
Moreover, if $\zeta_{\Omega,H}(s)$ has a pole of order $b\in\N$ at $s=1$, then the Cauchy residue theorem yields
\[
\mathbb{E}_{\Omega,H,X}[F] =
c \cdot \frac{P_{f,\zeta}(\log X)}{Q_{\zeta}(\log X)} + {\rm error},
\]
where $P_{f,\zeta}(t),Q_{\zeta}(t) \in \C[t]$ are degree $b$ monic polynomials (at least when $c\ne0$, which is the case relevant for our applications).
In particular, if the error term decays faster than any power of $\log X$ as $X\to \infty$, for example $O\left(e^{-\varepsilon\sqrt{\log X}}\right)$, then the rate of convergence of $\mathbb{E}_{\Omega,H,X}[F] \to c$ is explicitly determined by the algebraic expectation of the normalized Mellin transform of $F$.

We prove \cref{thm:EK criterion for ideals} using this explicit polynomial structure in the algebraic expectation operator $\widehat{\mathbb{E}}_{\zeta,X}[f]$, which is a $\C$-linear functional on the algebra of holomorphic germs at $s=1$ that is \textit{unital}, meaning that $\widehat{\mathbb{E}}[1] = 1$.
These properties arise more generally for algebraic probability spaces, where many of the features of classical probability theory, such as linearity of expectation, conditioning, and moments, still make sense.

Algebraic probability spaces are a standard idea used in the study of free probability and operator algebras; see for example \cite{NicaSpeicherLecturesOnFreeProbability}.
For our applications, we do not impose the condition that probabilities are non-negative real numbers, as we often decompose the algebraic analog of a counting measure into complex measures using character orthogonality relations, hence we allow ourselves to work with complex (or \textit{virtual}) probabilities.

\subsection{Algebraic probability theory: definitions}
\label{subsec:algebraic probability theory}

We now define the objects used throughout the probabilistic arguments of the paper.

\begin{definition}
	A \textit{(virtual, commutative) algebraic probability space} $\mc{A}$ is an associative commutative unital $\C$-algebra $\mc{A}_0$ with a linear functional $\widehat{\mathbb{E}} : \mc{A}_0 \to \C$ that is unital ($\widehat{\mathbb{E}}[1] = 1$).
\end{definition}

\begin{remark}
	We refer to elements of an algebraic probability space as random variables, and we refer to $\widehat{\mathbb{E}}$ as an expectation operator.
\end{remark}

\begin{definition}
	\label{def:Tauberian mixture for ideals}
	A \textit{Tauberian pre-mixture} $\mc{T}$ consists of the following data:
	\begin{itemize}
		\item a finite set $\Lambda = \Lambda(\mc{T})$, which we call the \textit{set of parameters};
		\item a positive real number $s_0 = s_0(\mc{T}) > 0$;
		\item a natural number $b(\mc{T}) \in \N$;
		\item for each $\lambda \in \Lambda$, a formal Dirichlet series $D_\lambda(s) = D_\lambda(\mc{T};s)$.
	\end{itemize}
	Let $k$ be a number field.
	A Tauberian pre-mixture $\mc{T}$ is a \textit{Tauberian mixture over $k$} if its data satisfy the following properties:
	\begin{itemize}
		\item for each $\lambda \in \Lambda(\mc{T})$, the formal Dirichlet series $D_\lambda(\mc{T};s)$ is absolutely convergent (hence holomorphic) in the region $\Re(s) > s_0(\mc{T})$;
		\item $D_\lambda(s)$ admits a meromorphic continuation to the region $\Re(s) \ge s_0$, and the only allowed singularity of $D_{\lambda}(s)$ in this region is a pole of order at most $b(\mc{T})$ at $s = s_0$;
		\item there is \textit{no cancellation of poles}, which states that
		\[
		\sum_{\lambda \in \Lambda} D_\lambda(s)
		\]
		has a pole of order $b$ at $s = s_0$.
	\end{itemize}
\end{definition}

\begin{construction}[Algebraic probability space of a Tauberian mixture]
	\label{construction:algebraic probability space associated to Tauberian mixture}
	Let
	\[
	\mc{T} = \mc{T}(\Lambda,s_0,b,(D_\lambda)_{\lambda \in \Lambda})
	\]
	be a Tauberian mixture, and $X_0 = X_0(\mc{T}) \ge 10$ a sufficiently large parameter (depending on $\mc{T}$) so that the following construction is well-defined.
	For every $X \ge X_0$, the \textit{algebraic probability space $\mc{A}(\mc{T}) = (\mc{A}_0(\mc{T}),\widehat{\mathbb{E}}_{\mc{T},X})$ associated to a Tauberian mixture $\mc{T}$} is constructed as follows.
	The associative commutative unital $\C$-algebra is given by
	\[
	\mc{A}_0(\mc{T}) := \bigoplus_{\lambda \in \Lambda(\mc{T})} \mc{H}_{\lambda},
	\]
	where each $\mc{H}_{\lambda}$ is a copy of the algebra of holomorphic germs based at $s = s_0(\mc{T})$.
	We write an element of $\mc{A}_0(\mc{T})$ as a $\Lambda$-tuple of holomorphic germs $\boldsymbol{f}(s) = (f^{\lambda}(s))_{\lambda \in \Lambda}$.
	The subalgebra of locally constant germs $\C^{\Lambda} \subset \mc{A}_0(\mc{T})$ has a unital $\C$-linear \textit{conditional expectation operator}
	\[
	\widehat{\mathbb{E}}_{\mc{T},X}[-\mid\Lambda] : \mc{A}_0(\mc{T}) \to \C^{\Lambda}
	\]
	that is defined for each parameter $\lambda \in \Lambda$ via (see \cref{rem:convention for measure zero conditional expectation})
	\[
	\widehat{\mathbb{E}}_{\mc{T},X}[\boldsymbol{f}(s) \mid \lambda] := 
	\begin{cases}
		\dfrac{\underset{s = s_0}{\Res}\Big( f^\lambda(s) D_\lambda(s) X^s/s \Big)}{\underset{s = s_0}{\Res}\Big( D_\lambda(s) X^s/s \Big)} & D_{\lambda}(s) \text{ has a pole at } s=s_0, \\
		f^{\lambda}(s_0) & \text{otherwise},
	\end{cases}
	\]
	where $\underset{s = s_0}{\Res}$ denotes the complex-analytic residue at $s = s_0$.
	Finally, we define the unital $\C$-linear \textit{expectation operator}
	\[
	\widehat{\mathbb{E}}_{\mc{T},X}[\boldsymbol{f}(s)] :=
	\frac{\underset{s = s_0}{\Res}\Big(\sum_{\lambda \in \Lambda} f^{\lambda}(s) D_\lambda(s) X^s/s \Big)}{\underset{s = s_0}{\Res}\Big( \sum_{\lambda \in \Lambda} D_{\lambda}(s) X^s/s \Big)}.
	\]
	The operator $\widehat{\mathbb{E}}_{\mc{T},X}$ satisfies the \textit{law of total expectation}:
	\[
	\widehat{\mathbb{E}}_{\mc{T},X}[\boldsymbol{f}(s)] = \sum_{\lambda \in \Lambda}  \widehat{\mathbb{E}}_{\mc{T},X}[\boldsymbol{f}(s) \mid \lambda] \cdot 
	\frac{\underset{s = s_0}{\Res}\Big( D_\lambda(s) X^s/s \Big)}{\underset{s = s_0}{\Res}\Big( \sum_{\lambda' \in \Lambda} D_{\lambda'}(s) X^s/s \Big)}.
	\]
\end{construction}

\begin{remark}
	\label{rem:residue is a polynomial in X}
	For $X_0 = X_0(\mc{T}) \ge 10$ sufficiently large, $\lambda \in \Lambda(\mc{T})$, and $X\ge X_0$ the complex numbers
	\[
	w_{\lambda}(X) := \frac{\underset{s = s_0}{\Res}\Big( D_\lambda(s) X^s/s \Big)}{\underset{s = s_0}{\Res}\Big( \sum_{\lambda' \in \Lambda} D_{\lambda'}(s) X^s/s \Big)}
	\]
	are well-defined and satisfy $\sum_{\lambda \in \Lambda} w_\lambda(X) = 1$.
	Moreover, one has that the limits $w_{\lambda}(\infty) := \lim\limits_{X\to\infty} w_\lambda(X)$ exist because $\mc{T}$ has no cancellation of poles.
	In fact, the rate of convergence is $O(1/\log X)$ because the numerator/denominator of the above expression can be written as $X^{s_0(\mc{T})}$ times a polynomial of degree at most/exactly $b(\mc{T})-1$ in the variable $\log X$.
\end{remark}

\begin{remark}
	\label{rem:convention for measure zero conditional expectation}
	If $\lambda \in \Lambda$ is such that $D_{\lambda}(s)$ does \textit{not} have a pole at $s = s_0$, then conditioning on $\lambda$ behaves like conditioning on a measure zero event in classical probability theory.
	In particular, the convention of defining $\widehat{\mathbb{E}}_{\mc{T},X}[\boldsymbol{f}(s) \mid \lambda] := f^{\lambda}(s_0)$ when $D_{\lambda}(s)$ does not have a pole at $s=s_0$ is arbitrary, and does \textit{not} play a role in our analysis of algebraic probabilities, expectations, moments, or cumulants used to prove the results in this paper.
\end{remark}

\begin{definition}
	Given a Tauberian mixture $\mc{T}$ over a number field $k$, we define the set
	\[
	\Lambda^+(\mc{T}) := \{\lambda \in \Lambda(\mc{T}) : \lim_{X \to \infty} w_\lambda(X) \ne 0\}.
	\]
\end{definition}

\subsubsection{The Bernoulli random variables in our algebraic probability space}

Given a Tauberian mixture $\mc{T}$ over $k$, we would like to assign for all but finitely many primes $\mfp$ an element $\boldsymbol{f}_\mfp(s) \in \mc{A}_0(\mc{T})$ that approximates the normalized Mellin transform of the indicator function of some local event at $\mfp$, such as $\mfp$ ramifying in a $G$-extension of $k$ for some finite group $G$.
Thus, $\boldsymbol{f}_\mfp(s)$ should be thought of as a Bernoulli random variable at the prime $\mfp$.
A collection of Bernoulli random variables is often called a Bernoulli process, and we give an algebraic probability definition as follows.

\begin{definition}[Bernoulli process]
	\label{def:Bernoulli process for ideals}
	Let $k$ be a number field.
	A \textit{Bernoulli process over $k$}, denoted $\mc{B}$, is the following tuple of data:
	\begin{itemize}
		\item $S = S(\mc{B})$ is a finite set of primes of $k$;
		\item $\mc{T} = \mc{T}(\mc{B})$ is a Tauberian mixture over $k$;
		\item $X_0 = X_0(\mc{T}) \ge 10$ is a valid parameter in \cref{construction:algebraic probability space associated to Tauberian mixture};
		\item for every prime $\mfp \notin S$ of $k$, $\boldsymbol{f}_{\mfp}(s) = \boldsymbol{f}_{\mfp}(\mc{B};s)$ is an element of $\mc{A}_0(\mc{T})$, which we refer to as the \textit{Bernoulli random variable associated to $\mfp$}.
	\end{itemize}
\end{definition}

\subsubsection{The moments of our random variables}

The point of packaging all of the data of our random model into the definition of a Bernoulli process $\mc{B}$ is so that we can clearly state all of the dependencies going into the following definition of the moments of our random variables.

\begin{definition}
	Let $k$ be a number field, and
	\[
	\mc{B} =
	\mc{B}\Big(S,\mc{T}\Big(\Lambda,s_0,b,(D_\lambda(s))_{\lambda \in \Lambda}\Big),X_0,(\boldsymbol{f}_{\mfp}(s))_{\mfp \notin S}\Big)
	\]
	a Bernoulli process over $k$.
		For every $X \ge X_0$ and tuple $\eu{P} = (\mfp_1,\ldots,\mfp_m)$ of primes of $k$ avoiding $S$ we define the \textit{algebraic $\eu{P}$-moment}
		\begin{equation}
		\label{eq:algebraic Sigma moment}
		\widehat{M}_{\mc{B},X}(\eu{P}) :=
		\widehat{\mathbb{E}}_{\mc{T},X}\left[\prodflat_{\mfp \in \eu{P}}\boldsymbol{f}_{\mfp}(s)\right],
		\end{equation}
		and for every $\lambda \in \Lambda$ we also define the \textit{$\lambda$-conditional moment}
		\[
		\widehat{M}_{\mc{B},X}(\eu{P} \mid \lambda) := \widehat{\mathbb{E}}_{\mc{T},X}\left[\prodflat_{\mfp \in \eu{P}}\boldsymbol{f}_{\mfp}(s) 
		\, \Bigg| \, \lambda\right].
		\]
	(Recall our convention for products over tuples in \cref{subsec:notation}.)
\end{definition}

\subsection{Reasonability}
\label{subsec:reasonability}

Our discussion in \cref{subsec:algebraic probability theory} ultimately led us to produce a collection of moments $(\widehat{M}_{\mc{B},X}(\eu{P}))_{\eu{P}}$ associated to a Bernoulli process $\mc{B}$ representing a collection of Bernoulli random variables $f_\mfp(\mc{B};s)$ indexed by primes $\mfp \notin S(\mc{B})$ of our number field $k$.
In our application, each Bernoulli random variable $f_\mfp(\mc{B};s)$ corresponds to the event that $\mfp$ divides a random discriminant $\disc(K/k) \in \mf{N}_k$, where $K/k$ is a random $G$-extension.
This motivates the following definition of a reasonable sequence of ideals of $\O_k$, i.e.\ one where we provably obtain a central limit theorem for the number of prime divisors via \cref{thm:EK criterion for ideals}.

\begin{definition}[Well-mixed Bernoulli process]
	\label{def:well-mixed Bernoulli process}
	Let $k$ be a number field and $\mc{B}$ a Bernoulli process over $k$.
	We say that $\mc{B}$ is \textit{well-mixed} if the following two conditions hold.
	\begin{enumerate}
		\item[(i)] (Mixture uniformity) One has for every $\lambda \in \Lambda(\mc{T})$ and $r \in \Z_{\ge 0}$ that
		\[
		\frac{d^r}{ds^r} f_\mfp^\lambda(s) \bigg|_{s=s_0}
		\ll_r \frac{(\log \NN\mfp)^r}{\NN\mfp}.
		\]
		\item[(ii)] (Leading order conditional probabilities are similar) There exists a positive constant $\alpha>0$ such that the following holds.
		For all $\mfp \notin S$ and $\lambda,\lambda' \in \Lambda^+(\mc{T})$ one has that
		\[
		|f_\mfp^\lambda(s_0) - f_\mfp^{\lambda'}(s_0)| \ll
		\frac{1}{(\NN\mfp)^{1+\alpha}}.
		\]
	\end{enumerate}
\end{definition}

\begin{definition}[Reasonable ordered sequence of ideals]
	\label{def:reasonable sequence for ideals}
	Let $k$ be a number field, and $(\msF,H)$ an ordered sequence of ideals in $\O_k$.
	We say that $(\msF,H)$ is \textit{reasonable} if the following two conditions hold.
	\begin{itemize}
		\item[(i)] (Large prime tails are sparse) For all $X \ge 10$ sufficiently large, $0 < \psi \le 1$ (possibly depending on $X$), and $\mfn \in \msF_H(X)$ one has that
		\[
		\#\{\mfp \subset \O_k \textnormal{ prime} : \mfp \mid \mfn,\ \NN\mfp > X^\psi\}
		\ll_{\msF,H} \frac{1}{\psi}.
		\]
		\item[(ii)] (Approximation by $\mc{B}$) There exists a well-mixed Bernoulli process $\mc{B}$ such that the following holds.
		For each $m\in\N$, there exist positive constants $A_m,\delta_m > 0$ such that for every tuple $\eu{P} = (\mfp_1,\ldots,\mfp_m)$ of primes of $k$ avoiding $S(\mc{B})$ and $X \ge X_0(\mc{B})$ one has that
		\[
		\frac{\#\{\mfn \in \msF_H(X) : \mfp_1,\ldots,\mfp_m \text{ divide } \mfn\}}{\#\msF_H(X)}
		= \widehat{M}_{\mc{B},X}(\eu{P}) + O_{\msF,H,\mc{B}}\left(\prod_{\mfp \in \eu{P}}(\NN\mfp)^{A_m} \cdot \exp\left(
			- \frac{\log X}{(\log \log X)^{1/2-\delta_m}}
		\right)
		\right).
		\]
	\end{itemize}
\end{definition}

\section{Heights and ramification type}
\label{sec:heights}

In this section, we develop our notation for counting number fields.
Much of the terminology is borrowed from \cite{LoughranSantensMalleBrauer,LoughranSantensSurvey}, although there are small technical differences in our jargon.

\begin{notation}
	Throughout this section, we fix a number field $k$, a nontrivial finite group $G$, and for every place $v$ of $k$ a choice of decomposition group $D_v \subset \Gamma_k$.
	For every non-archimedean place $v$, there is also the inertia subgroup $I_v \subset D_v$ with tame quotient $I_v \twoheadrightarrow I_v^{\rm tame}$.
	In this case, there is a canonical identification $I_v^{\rm tame} \cong \widehat{\Z}(1)$ where $\widehat{\Z}(1) := \varprojlim_n \mu_n$.
\end{notation}

\subsection{Heights}

\begin{definition}
	\label{def:weight function}
	A \textit{weight function} $w:G \to \Z_{\ge 0}$ is a class function such that $w(g) = 0$ if and only if $g$ is the identity element of $G$, and $w(g^e) = w(g)$ for all $e \in (\Z/\exp(G)\Z)^\times$.
\end{definition}

\begin{definition}
	\label{def:Malle's a and b constants}
	Let $w:G\to\Z_{\ge 0}$ be a weight function.
	We define \textit{Malle's $a$-constant} $a(G,w) \in \N$ to be the minimal positive value attained by $w$.
	We define \textit{Malle's $b$-constant} $b(k,G,w) \in \N$ to be the number of orbits of conjugacy classes $c \subset G$ satisfying $w|_c = a(G,w)$ under the action $c \mapsto c^{{\rm cyc}(\sigma)}$ where ${\rm cyc}:\Gal(k(\zeta_{\exp(G)})/k) \to (\Z/\exp(G)\Z)^\times$ is the cyclotomic character.
\end{definition}

\begin{definition}
	\label{def:height function}
	A \textit{height function $H:\Hom(\Gamma_k,G) \to \R_{>0}$} with weight $w:G \to \Z_{\ge 0}$ is a function that can be constructed as follows:
	\begin{itemize}
		\item choose a finite set of places $T$ of $k$ containing all places dividing $|G|\infty$;
		\item for every $v \in T$, choose an arbitrary function $H_v:\Hom(D_v,G) \to \R_{> 0 }$;
		\item for every place $v\notin T$, define $H_v : \Hom(D_v,G) \to \R_{>0}$ via
		\[
		H_v(\varphi_v) := (\NN v)^{w(\varphi_v(\sigma_v))}
		\]
		where $\sigma_v$ is an arbitrary generator of the tame quotient of local inertia at $v$;
		\item finally, set
		\[
		H(\varphi) := \prod_v H_v(\varphi|_{D_v})
		\]
		where $v$ varies over all places of $k$.
	\end{itemize}
	Because $w$ is invariant under all invertible powers, the local factor $H_v(\varphi_v)$ for $v \notin T$ is independent of the choice of generator $\sigma_v$.
	The minimal subset $T$ that can be used to construct $H$ is called the \textit{exceptional set} of $H$, and is denoted ${\rm Exc}(H)$.
\end{definition}

\begin{remark}
	In the terminology of \cite{LoughranSantensMalleBrauer,LoughranSantensSurvey}, our height functions are precisely the \textit{$\widehat{\Z}^\times$-invariant} height functions.
\end{remark}

\begin{remark}
	Wood's notion of a counting function \cite[\S2.1]{WoodLocalAbelianProbabilities} extended to non-abelian groups essentially coincides with our notion of a height function.
\end{remark}

\begin{remark}
	A height function $H$ can be reconstructed from the data of: its weight $w$, the exceptional set ${\rm Exc}(H)$, and the auxiliary local functions $H_v$ for all $v\in {\rm Exc}(H)$.
	When we say that ``$H$ is a height function,'' we often implicitly remember this data.
\end{remark}

\subsection{Ramification type}

\begin{definition}
	\label{def:ramification above Sigma}
	Let $\varphi \in \Sur(\Gamma_k,G)$, and $\eu{P}$ a finite set of primes of $k$.
	We say that $\varphi$ is \textit{ramified above $\eu{P}$} if for all $\mfp \in \eu{P}$ the restriction of $\varphi$ to the inertia subgroup $I_{\mfp} \subset D_{\mfp} \subset \Gamma_{k}$ is nontrivial.
\end{definition}

\begin{remark}
	If the finite set of primes $\eu{P}$ avoids all primes dividing $|G|$, then $\varphi$ is ramified above $\eu{P}$ if and only if, for every $\mfp \in \eu{P}$ and topological generator $\sigma_{\mfp}$ of $I_{\mfp}^{\rm tame}$, the element $\varphi(\sigma_{\mfp}) \in G$ is nontrivial.
\end{remark}

\begin{definition}
	\label{def:sectors}
	A \textit{sector $\mf{S}$ of $G$ defined over $k$}, or simply a \textit{ramification type}, is a subset of the $\Gamma_k$-set $G(-1) := \Hom(\widehat{\Z}(1),G)$ that is nonempty and minimal with respect to the property of being closed under $G$-conjugation and closed under the $\Gamma_k$-action.
	We let $\mc{S}(k,G)$ denote the set of $k$-sectors of $G$, and $\mc{S}^*(k,G) := \mc{S}(k,G)\backslash\{1\}$ is the set of \textit{nontrivial $k$-sectors of $G$}.
\end{definition}

\begin{definition}
	\label{def:ramification type}
	Let $\varphi \in \Sur(\Gamma_k,G)$, $\mf{S} \in \mc{S}^*(k,G)$, and $\mfp$ a prime of $k$ that does not divide $|G|$.
	We say that $\varphi$ has \textit{ramification type $\mf{S}$ above $\mfp$} if $\varphi$ restricted to the inertia subgroup $I_{\mfp}$ factors through an element of $\mf{S} \subset G(-1)$ using the canonical identification $I_{\mfp}^{\rm tame} \cong \widehat{\Z}(1)$.
\end{definition}

\begin{definition}
	\label{def:minimal weight sectors}
	Let $w:G \to \Z_{\ge 0}$ be a weight function.
	We define the set of \textit{minimal weight $k$-sectors of $G$ with respect to $w$} to be
	\[
	\mc{MS}(k,G,w) :=
	\left\{
	\mf{S} \in \mc{S}^*(k,G) :
	w(g) = a(G,w) \ \forall g \in \mf{S}
	\right\}.
	\]
\end{definition}

\begin{remark}
	\label{rem:number of sectors is Malle b}
	The quantity $\#\mc{MS}(k,G,w)$ is equal to the number of cyclotomic orbits of minimal weight conjugacy classes, hence $\#\mc{MS}(k,G,w) = b(k,G,w)$.
\end{remark}

\section{Erd\H{o}s--Kac for abelian extensions}
\label{sec:Erdos Kac abelian extensions}

In this section, we use previous results on Malle's conjecture for abelian extensions due to Wright \cite{WrightAbelianExtensions}, Wood \cite{WoodLocalAbelianProbabilities}, Frei--Loughran--Newton \cite{FreiLoughranNewtonPowerSavingsAbelianExtensions}, and Alberts \cite{AlbertsPowerSavingsAbelianExtensions,AlbertsAveragedInputTauberianTheorems} in order to deduce a ``uniform equidistribution theorem'' for abelian extensions (\cref{prop:quantitative equidistribution for abelian extensions}).
We use these results in tandem with \cref{thm:EK criterion for ideals} in order to obtain an Erd\H{o}s--Kac theorem (\cref{thm:Erdos Kac abelian extensions}) that recovers \cref{thm:Erdos Kac abelian counting function}.

\begin{theorem}[Erd\H{o}s--Kac for abelian extensions]
	\label{thm:Erdos Kac abelian extensions}
	Let $k$ be a number field, $G$ a nontrivial finite abelian group, and $H:\Hom(\Gamma_k,G) \to \R_{>0}$ a height function with weight $w$.
	Let $\omega:\Sur(\Gamma_k,G) \to \R$ denote the function
	\[
	\omega(\varphi) := \#\{\mfp \subset \O_k \textnormal{ prime} : \varphi \textnormal{ ramified above }\mfp\}.
	\]
	For every interval $I \subset \R$ one has that
	\begin{multline*}
		\lim_{X\to\infty}
		\frac{1}{\#\{\varphi: H(\varphi) \leq X\}} \cdot
		\#\left\{
		\varphi : H(\varphi) \leq X, \frac{\omega(\varphi) - b(k,G,w) \log\log X}{\sqrt{b(k,G,w) \log\log X}} \in I
		\right\} \\ = 
		\frac{1}{\sqrt{2\pi}} \int_I e^{-t^2/2} \,dt,
	\end{multline*}
	where $\varphi$ varies over $\Sur(\Gamma_k,G)$.
\end{theorem}

\subsection{Wright's adelic theory}

\begin{notation}
	\label{notation:Erdos Kac abelian extensions}
	In the remainder of this section, we fix the following notation.
	We fix a number field $k$, for every place $v$ of $k$ a fixed completion $k_v$ with ring of integers $\O_v$ and decomposition group $D_v \subset \Gamma_k$, a nontrivial finite abelian group $G$ (equipped with the discrete topology), and a height function $H:\Hom(\Gamma_k,G) \to \R_{>0}$ with weight $w$.
	For every finite set of places $T$ of $k$ containing the archimedean places, we let $\O_T := \{x \in k : \forall v \notin T, \ v(x) \ge 0\}$.
	We fix a finite set of places $S$ such that $S$ contains all archimedean places, the set of primes dividing $|G|$, the set of primes of absolute norm at most $|G|^2$, the set ${\rm Exc}(H)$, and a set of primes generating the class group of $k$.
\end{notation}

We define the \textit{$S$-ramified idele group of $k$} to be the set
\[
J_S := \prod_{v \in S} k_v^\times \times \prod_{v \notin S} \O_v^\times
\]
equipped with the product topology.
It follows from the global reciprocity law in class field theory that there exists a canonical surjection of topological groups $\theta_k : J_S/\O_S^\times \twoheadrightarrow \Gamma_k^{\rm ab}$ whose kernel is the identity component of $J_S/\O_S^\times$.
In particular, $\theta_k$ gives rise to a canonical bijection between continuous homomorphisms $\varphi:\Gamma_k \to G$ and continuous characters $\chi : J_S \to G$ satisfying $\chi|_{\O_S^\times} = 0$.
The reciprocity map $\theta_k$ also satisfies a compatibility with local class field theory, which for our purposes upgrades the above bijection to include local conditions.
We state this as the following lemma.

\begin{lemma}
	\label{lem:class field theory bijection with local conditions}
	Let $k,G,H,S$ be as in \cref{notation:Erdos Kac abelian extensions}.
	For every finite set $\eu{P}$ of primes of $k$ avoiding $S$, there is a canonical bijection
	\[
	\left\{
	\begin{matrix}
		\textnormal{$\varphi \in \Hom(\Gamma_k,G)$ such that} \\
		\textnormal{$\varphi$ is ramified above $\eu{P}$}
	\end{matrix}
	\right\}
	\longleftrightarrow
	\left\{
	\begin{matrix}
		\textnormal{$\chi \in \Hom(J_S,G)$ such that} \\
		\textnormal{$\chi|_{\O_S^\times} = 0$ and $\forall \mfp \in \eu{P}, \ \chi(\zeta_{\mfp}') \neq 1$}
	\end{matrix}
	\right\}
	\]
	where $\zeta_{\mfp}' \in \O_\mfp^\times \subset J_S$ is any generator of the prime-to-the-residue-characteristic roots of unity in $k_\mfp$.
\end{lemma}

\begin{proof}
	This follows from \cite[Lem.\ 2.7]{WoodLocalAbelianProbabilities} and the statements of class field theory.
\end{proof}

\cref{lem:class field theory bijection with local conditions} motivates the following definition of ramification for characters.

\begin{definition}
	Let $k,G,S$ be as in \cref{notation:Erdos Kac abelian extensions}, and $\eu{P}$ a finite set of primes of $k$ avoiding $S$.
	We say that $\chi \in \Hom(J_S,G)$ is \textit{ramified above $\eu{P}$} if for all $\mfp \in \eu{P}$ one has that $\chi(\zeta_{\mfp}') \ne 1$ where $\zeta_{\mfp}' \in \O_{\mfp}^\times$ is an arbitrary generator of the prime-to-the-residue-characteristic roots of unity in $k_{\mfp}$.
\end{definition}

\cref{lem:class field theory bijection with local conditions} motivates the following construction of a height function for characters.

\begin{construction}
	Let $k,G,H,S$ be as in \cref{notation:Erdos Kac abelian extensions}.
	Using local class field theory, we may view $H$ as a function $H':\Hom(J_S,G) \to \R_{>0}$.
	Indeed, for every place $v$ of $k$ we have by local class field theory a function $H_v':\Hom(k_v^\times,G) \xrightarrow{\sim} \Hom(D_v,G) \xrightarrow{H_v} \R_{>0}$, and we define $H' := \prod_v H_v' \circ \res_v$ where $\res_v:\Hom(J_S,G) \to \Hom(k_v^{\times},G)$ is the local restriction map at $v$.
\end{construction}

The condition that $\chi|_{\O_S^\times}$ vanishes can be computed in terms of the perfect pairing of Pontryagin dual groups
\begin{equation}
\label{eq:perfect pairing S integral}
\Hom(\O_S^\times,G) \times (\O_S^\times \otimes G^\vee) \to \Q/\Z, \qquad 
(\psi,\varepsilon \otimes f) \mapsto f(\psi(\varepsilon)).
\end{equation}
By Dirichlet's unit theorem, the groups $\Hom(\O_S^\times,G)$ and $\O_S^\times \otimes G^\vee$ are finite.
In fact, because $\O_S^\times \to J_S$ is an injection we can recover the above pairing as a sum of local perfect pairings of Pontryagin dual groups
\[
\langle-,-\rangle_v:
\Hom(k_v^\times,G) \times (k_v^\times \otimes G^\vee) \to \Q/\Z, \qquad \langle\chi_v,\varepsilon_v\otimes f\rangle_v := f(\chi_v(\varepsilon_v))
\]
over all places $v$ of $k$.
We formalize this as follows.

\begin{definition}
	Let $k,G,S$ be as in \cref{notation:Erdos Kac abelian extensions}.
	We define the \textit{obstruction pairing}
	\[
	\langle-,-\rangle : \Hom(J_S,G) \times (\O_S^\times\otimes G^\vee) \to \Q/\Z
	\]
	via
	\[
	\langle (\chi_v)_v ,\varepsilon\rangle :=
	\sum_v \, \langle \chi_v, \varepsilon\rangle_v
	\]
	where $v$ varies over all places of $k$.
	Indeed, this is well-defined because continuity enforces $\chi_v$ to be trivial for all but finitely many $v$.
\end{definition}

\begin{lemma}
	Let $k,G,S$ be as in \cref{notation:Erdos Kac abelian extensions}.
	Then the obstruction pairing $\langle-,-\rangle$ has left kernel equal to $\Hom(J_S/\O_S^\times,G)$.
\end{lemma}

\begin{proof}
	The pairing $\langle-,-\rangle$ naturally factors through the perfect pairing \cref{eq:perfect pairing S integral}, so the left kernel of $\langle-,-\rangle$ is equal to the kernel of the restriction map $\Hom(J_S,G)\to\Hom(\O_S^\times,G)$.
\end{proof}

\begin{definition}
	Let $k,G,S$ be as in \cref{notation:Erdos Kac abelian extensions}.
	We define the \textit{obstruction group} ${\rm Obs}(k,G,S) := \O_S^\times \otimes G^\vee$.
\end{definition}

\begin{remark}
	\label{rem:abelian obstruction group is a Brauer group}
	We have chosen $S$ sufficiently large so that ${\rm Obs}(k,G,S) = H^1(\O_S,\Hom(G,\O_S^\times))$ is the algebraic Brauer group of $BG$ over $\Spec(\O_S)$.
	Comparing with the work of Loughran and Santens, the space $\Hom(J_S,G)$ resembles an $S$-integral version of a partial adelic space of $BG$ \cite[Def.\ 7.2]{LoughranSantensMalleBrauer}, the pairing $\langle-,-\rangle$ resembles an $S$-integral version of the Brauer--Manin pairing \cite[Lem.\ 7.5]{LoughranSantensMalleBrauer}, and the left kernel $\Hom(J_S/\O_S^{\times},G)$ resembles an $S$-integral version of the Brauer--Manin set \cite[\S8.5]{LoughranSantensMalleBrauer}.
\end{remark}

\begin{definition}
	\label{def:abelian ramification factors f mfp}
	Let $k,G,w,S$ be as in \cref{notation:Erdos Kac abelian extensions}.
	For every prime $\mfp \notin S$ of $k$ and $\beta \in {\rm Obs}(k,G,S)$, we define the \textit{$\beta$-twisted ramification factor at $\mfp$} to be the formal meromorphic function
	\[
	f_{\mfp}^{\beta}(s) = 
	f_{\mfp}^{G,\beta}(s) := \frac{\sum_{\substack{\chi_{\mfp} \in \Hom(\O_{\mfp}^\times,G) \\ \chi_{\mfp}(\zeta_{\mfp}') \ne 1}} 
		e^{2\pi i \langle \chi_\mfp,\beta\rangle_\mfp}
		(\NN \mfp)^{-w(\chi_{\mfp}(\zeta_{\mfp}')) s}}
	{
		\sum_{\chi_{\mfp} \in \Hom(\O_{\mfp}^\times,G)} 
		e^{2\pi i \langle \chi_\mfp,\beta\rangle_\mfp}
		(\NN \mfp)^{-w(\chi_{\mfp}(\zeta_{\mfp}')) s}
	},
	\]
	where $\zeta_{\mfp}'$ is an arbitrary generator of the prime-to-the-residue-characteristic roots of unity in $k_{\mfp}$.
	We also let $f_{\mfp}(s) := f_{\mfp}^{0}(s)$ where $0$ is the identity element of ${\rm Obs}(k,G,S)$.
\end{definition}

\begin{lemma}
	\label{lem: f mfp (s) are holomorphic}
	Let $k,G,S$ be as in \cref{notation:Erdos Kac abelian extensions} and $\beta \in {\rm Obs}(k,G,S)$.
	For every prime $\mfp \notin S$ of $k$, $f_\mfp^{\beta}(s)$ is holomorphic in the region $\Re(s) > 1/(a(G,w)+1)$.
\end{lemma}

\begin{proof}
	Let $\mfp$ be as in the lemma statement.
	It suffices to show that the denominator of $f_\mfp^{\beta}(s)$ is nonzero.
	Let $s\in \C$ with $\Re(s) > 1/(a(G,w)+1)$.
	Then pulling out the trivial character from the denominator and applying the reverse triangle inequality yields
	\[
	\left|
		\sum_{\chi_{\mfp} \in \Hom(\O_{\mfp}^\times,G)} 
		\frac{e^{2\pi i \langle \chi_\mfp,\beta\rangle_\mfp}}
		{(\NN \mfp)^{w(\chi_{\mfp}(\zeta_{\mfp}')) s}}
	\right|
	\ge 
	1 -  \frac{\#\Hom(\O_\mfp^\times,G)}{(\NN\mfp)^{a/(a+1)}}.
	\]
	Since $\#\Hom(\O_{\mfp}^\times,G) \le |G|$ and $a/(a+1) \ge 1/2$, the lower bound is at least $1-|G|/(\NN\mfp)^{1/2} > 0$.
\end{proof}

\begin{definition}
	Let $k,G,w,H,S$ be as in \cref{notation:Erdos Kac abelian extensions} and $\eu{P}$ a finite set of primes of $k$ avoiding $S$.
	For every $\beta \in {\rm Obs}(k,G,S)$, we define the \textit{$\beta$-twisted Wright series with ramification above $\eu{P}$} to be the formal Euler product
	\[
	{\rm W}_{k,G,H}^{\beta,\eu{P}}(s) :=
	\prod_{v \in S}
	\left(
	\sum_{\chi_v \in \Hom(k_v^\times,G)}
	\frac{e^{2\pi i \langle \chi_v,\beta_v\rangle_v}}{H_v'(\chi_v)^s}
	\right)
	\prod_{v \notin S} \left(
	\sum_{\chi_v \in \Hom(\O_v^\times,G)}
	\frac{e^{2\pi i \langle \chi_v,\beta_v\rangle_v}}{H_v'(\chi_v)^s}
	\right)
	\prodflat_{\mfp \in \eu{P}} f_{\mfp}^{\beta}(s).
	\]
	where $v$ varies over all places of $k$.
\end{definition}

\begin{definition}
	Let $k,G,H$ be as in \cref{notation:Erdos Kac abelian extensions}.
	For every subgroup $N \le G$, we define the \textit{restricted height} $H|_N : \Hom(\Gamma_k,N) \to \R_{>0}$ to be $\Hom(\Gamma_k,N) \to \Hom(\Gamma_k,G) \xrightarrow{H} \R_{>0}$.
\end{definition}

\begin{lemma}
	\label{lem:counting series as a finite sum of euler products abelian case}
	Let $k,G,H,S$ be as in \cref{notation:Erdos Kac abelian extensions} and $\eu{P}$ a finite set of primes of $k$ avoiding $S$.
	The following statements hold.
	\begin{enumerate}
		\item One has an equality of formal Dirichlet series
		\[
		\sum_{\substack{\varphi \in \Hom(\Gamma_k,G) \\ \varphi \textnormal{ is ramified above $\eu{P}$}}} \frac{1}{H(\varphi)^s} = 
		\frac{1}{\#{\rm Obs}(k,G,S)}
		\sum_{\beta \in {\rm Obs}(k,G,S)} {\rm W}_{k,G,H}^{\beta,\eu{P}}(s).
		\]
		\item One has an equality of formal Dirichlet series
		\[
		\sum_{\substack{\varphi \in \Sur(\Gamma_k,G) \\ \varphi \textnormal{ is ramified above $\eu{P}$}}} \frac{1}{H(\varphi)^s} = 
		\sum_{N \le G}
		\frac{\mu(N,G)}{\#{\rm Obs}(k,N,S)}
		\sum_{\beta \in {\rm Obs}(k,N,S)} {\rm W}_{k,N,H|_N}^{\beta,\eu{P}}(s)
		\]
		where $\mu(N,G)$ is the upper interval Möbius function on the lattice of subgroups $N \le G$.
	\end{enumerate}
\end{lemma}

\begin{proof}[Proof of \cref{lem:counting series as a finite sum of euler products abelian case}]
	The second statement for $G$ follows from applying Möbius inversion to the first statement with data $(k,N,H|_N,S,\eu{P})$ as $N\le G$ varies over all subgroups.
	Therefore, it suffices to prove the first statement.
	As $v$ varies over all places of $k$, and for all $\mfp\in\eu{P}$ we let $\zeta_{\mfp}'$ be an arbitrary generator of the prime-to-the-residue-characteristic roots of unity in $k_{\mfp}$, we have that
	\begin{align*}
		\sum_{\substack{\varphi \in \Hom(\Gamma_k,G) \\ \varphi \textnormal{ is ramified above $\eu{P}$}}} \frac{1}{H(\varphi)^s}
		&= 
		\sum_{\substack{\chi \in \Hom(J_S,G) \\ \chi \text{ is ramified above $\eu{P}$,} \\ \chi|_{\O_S^\times} = 0}} \frac{1}{H'(\chi)^s}
		\qquad \text{ by \cref{lem:class field theory bijection with local conditions}}
		\\ &= 
		\sum_{\substack{\chi \in \Hom(J_S,G)\\ \chi \text{ is ramified above $\eu{P}$}}} 
		\frac{1}{H'(\chi)^s}
		\frac{1}{\#{\rm Obs}(k,G,S)} \sum_{\beta \in {\rm Obs}(k,G,S)} e^{2\pi i \langle \chi,\beta\rangle}
		\\ &=
		\sum_{\substack{(\chi_v)_v \in \Hom(J_S,G)\\ \forall \mfp \in \eu{P}, \  \chi_{\mfp}(\zeta_{\mfp}') \ne 1}} 
		\frac{1}{\#{\rm Obs}(k,G,S)} \sum_{\beta \in {\rm Obs}(k,G,S)} 
		\prod_v \frac{e^{2\pi i \langle \chi_v,\beta\rangle_v}}{H_v'(\chi_v)^s}
		\\
		= 
		\frac{1}{\#{\rm Obs}(k,G,S)} \sum_{\beta \in {\rm Obs}(k,G,S)} 
		\prod_{v \in S} &\left(\sum_{\chi_v \in \Hom(k_v^\times,G)} \frac{e^{2\pi i \langle\chi_v,\beta\rangle_v}}{H_v'(\chi_v)^s}\right)
		\prod_{v \notin S} \left( \sum_{\chi_v \in \Hom(\O_v^\times,G)} \frac{e^{2\pi i \langle\chi_v,\beta\rangle_v}}{H_v'(\chi_v)^s} \right)
		\prodflat_{\mfp \in \eu{P}} f_{\mfp}^{\beta}(s)
		\\ &=
		\frac{1}{\#{\rm Obs}(k,G,S)} \sum_{\beta \in {\rm Obs}(k,G,S)} 
		{\rm W}_{k,G,H}^{\beta,\eu{P}}(s).
		\end{align*}
		This finishes the proof.
\end{proof}

\subsection{Basic tail estimate}

The following lemma is exactly what is needed to verify the condition appearing in \cref{def:reasonable sequence for ideals}(i) in our applications.

\begin{lemma}[Tail estimate for ramified primes]
	\label{lem:restricted discriminants have sparse tails}
	Let $k,G,H,w$ be as in \cref{notation:Erdos Kac abelian extensions}.
	Then for all $X \ge 10$ sufficiently large, $\varphi \in \Sur(\Gamma_k,G)$ with $H(\varphi) \le X$, and $0 < \psi \le 1$ one has that
	\[
	\#\{\mfp \subset \O_k \textnormal{ prime} : \mfp \mid \disc(\varphi),\ \NN\mfp > X^\psi\}
	\ll_{k,G,H} \frac{1}{\psi}.
	\]
\end{lemma}

\begin{proof}
	Fix $\psi \in (0,1]$ and let
	\[
	U := \{\mfp \subset \O_k \textnormal{ prime} : \mfp \mid \disc(\varphi),\ \mfp \notin {\rm Exc}(H),\ \NN\mfp > X^\psi\}.
	\]
	For every $\mfp \in U$ we have
	\[
	H_\mfp(\varphi|_{D_\mfp}) \ge X^{a(G,w)\psi},
	\]
	so
	\[
	X \ge H(\varphi) \ge \prod_{\mfp \in U} H_\mfp(\varphi|_{D_\mfp}) \ge X^{a(G,w)\psi \cdot \#U}.
	\]
	Hence
	\[
	\#U \le \frac{1}{a(G,w)\psi}.
	\]
	Adding the finitely many primes in ${\rm Exc}(H)$ proves the claim.
\end{proof}

\subsection{Uniform equidistribution for abelian extensions}

We continue to use \cref{notation:Erdos Kac abelian extensions}.
In this subsection, we prove \cref{prop:quantitative equidistribution for abelian extensions}.

\begin{definition}
	Let $k,G,H,S$ be as in \cref{notation:Erdos Kac abelian extensions} and $\eu{P}$ a finite set of primes of $k$ avoiding $S$.
	We define the \textit{counting series with ramification above $\eu{P}$} to be the formal Dirichlet series
	\[
	\mc{N}_{k,G,H}^{\eu{P}}(s) := 
	\sum_{\substack{\varphi \in \Sur(\Gamma_k,G) \\ \varphi \textnormal{ is ramified above $\eu{P}$}}} \frac{1}{H(\varphi)^s}.
	\]
	We also define for all $X \ge 10$ the \textit{counting function with ramification above $\eu{P}$}
	\[
	{\rm N}_{k,G,H}^{\eu{P}}(X) := 
	\#\{\varphi \in \Sur(\Gamma_k,G) : H(\varphi) \le X, \text{$\varphi$ is ramified above $\eu{P}$}\}.
	\]
	We also let $\mc{N}_{k,G,H}(s) := \mc{N}_{k,G,H}^{\es}(s)$.
\end{definition}

\begin{lemma}
	\label{lem:meromorphic continuation Wright and counting series abelian}
	Let $k,G,H,S$ be as in \cref{notation:Erdos Kac abelian extensions} and $\eu{P}$ a finite set of primes of $k$ avoiding $S$.
	The following hold.
	\begin{enumerate}
		\item[(i)] For every subgroup $N \le G$ and every $\beta \in {\rm Obs}(k,N,S)$, the twisted Wright series ${\rm W}_{k,N,H|_N}^{\beta,\eu{P}}(s)$ admits a meromorphic continuation to the half-plane $\Re(s) > 1/(a(G,w)+1)$, and the only possible pole in this region is at $s = 1/a(G,w)$ of order at most $b(k,G,w)$.
		\item[(ii)] If $\sideset{}{^{\flat}_{\mfp \in \eu{P}}}{\prod} f_{\mfp}(s)$ does not vanish at $s = 1/a(G,w)$, then the counting series $\mc{N}_{k,G,H}^{\eu{P}}(s)$ is meromorphic in the region $\Re(s) > 1/(a(G,w)+1)$, and its only pole in this region is at $s = 1/a(G,w)$ of order $b(k,G,w)$.
		\item[(iii)] For every subgroup $N\le G$ and every $\beta \in {\rm Obs}(k,N,S)$, $W_{k,N,H|_N}^{\beta}(s)$ has a pole at $s=1/a(G,w)$ of order exactly $b(k,G,w)$ if and only if $N$ contains all minimal weight elements of $G$ with respect to $w$, and for all but finitely many primes $\mfp$ of $k$ avoiding $S$ one has that $\langle \beta, \chi_{\mfp} \rangle_{\mfp} = 0$ for all characters $\chi_{\mfp}: \O_{\mfp}^\times \to N$ such that $w(\chi_{\mfp}(\zeta_{\mfp}')) = a(G,w)$, where $\zeta_{\mfp}' \in \O_{\mfp}^{\times}$ is any generator of the prime-to-the-residue-characteristic roots of unity in $k_{\mfp}$.
	\end{enumerate}
\end{lemma}

\begin{proof}
	The first statement in the lemma is exactly the meromorphic continuation established for twisted abelian counting series in \cite[Thm.\ A.2, Thm.\ A.4, proof of Thm.\ 8.1]{AlbertsPowerSavingsAbelianExtensions}.

	For the second statement in the lemma, use \cref{lem:counting series as a finite sum of euler products abelian case}, the first statement in the lemma, and the fact that the largest poles do not cancel by Alberts--O'Dorney \cite[\S4]{AlbertsODorney}.
	Indeed, their results apply to our setting because our height functions are admissible and Frobenian in their framework.

	The third statement in the lemma is implied for example by Wood \cite[Lem.\ 2.10]{WoodLocalAbelianProbabilities} and Alberts \cite[Prop.\ 2.2]{AlbertsTwistedMalle}, although some tracing through definitions is required.
	Tavernier \cite[\S 2.4]{TavernierRestrictedRamificationAbelianCount} states that $W_{k,N,H|_N}^{\beta}(s)$ (in her notation $\widehat{f}_{H,N(-1)^*}(\beta;s)$) has a rightmost pole of order $b_{N(-1)^*}(H|_N;\beta) \allowbreak \le b(k,N,w|_N)$ at $s = 1/a(N,w|_N)$, with equality of $b$-constants if and only if for all but finitely many primes $\mfp$ of $k$ avoiding $S$ one has that
	\[
	\sum_{\substack{\chi_{\mfp}:\O_{\mfp}^{\times} \to N  \\ w(\chi_{\mfp}(\zeta_{\mfp}')) = a(N,w|N)}} e^{2\pi i \langle \chi_{\mfp},\beta \rangle_{\mfp}} = 
	\sum_{\substack{\chi_{\mfp}:\O_{\mfp}^{\times} \to N  \\ w(\chi_{\mfp}(\zeta_{\mfp}')) = a(N,w|N)}} 1.
	\]
	This proves the desired statement, since $a(N,w|_N) \ge a(G,w)$ and $b(k,N,w|_N) \le b(k,G,w)$ are both equalities if and only if $N$ contains all minimal weight elements of $G$ with respect to $w$. 
\end{proof}

\begin{proposition}
	\label{prop:quantitative equidistribution for abelian extensions}
	Let $k,G,H,S$ be as in \cref{notation:Erdos Kac abelian extensions}.
	There exist positive constants $A,\delta > 0$, depending only on $k,G,H$, such that for every finite set $\eu{P} = \{\mfp_1,\ldots,\mfp_m\}$ of primes of $k$ avoiding $S$ and $X\ge 10$ one has that
	\[
	{\rm N}_{k,G,H}^{\eu{P}}(X) =  
	\underset{s=1/a(G,w)}{\Res} \left(\mc{N}_{k,G,H}^{\eu{P}}(s) X^s/s\right)
	+ O_{k,G,H}((\NN\mfp_1\dotsi\NN\mfp_m)^A X^{1/a(G,w)-\delta}).
	\]
\end{proposition}

Several weaker versions of \cref{prop:quantitative equidistribution for abelian extensions} are explicitly stated in the literature.
Frei, Loughran, and Newton \cite[Thm.\ 1.7]{FreiLoughranNewtonPowerSavingsAbelianExtensions} proved that
\[
{\rm N}_{k,G,\disc}^{\eu{P}}(X) =  
\underset{s=1/a(G,w)}{\Res} \left(\mc{N}_{k,G,\disc}^{\eu{P}}(s) X^s/s\right)
+ O_{k,G,\eu{P}}(X^{1/a(G,w)-\delta}),
\]
where both the $\eu{P}$-dependence in the big-O term and the constant $\delta > 0$ are inexplicit.
Later, Alberts \cite[Thm.\ A.2]{AlbertsPowerSavingsAbelianExtensions} proved an explicit constant $\delta > 0$; however, the $\eu{P}$-dependence in the big-O term remained inexplicit.
More recently, Alberts' explicit Tauberian theorem \cite[Thm.\ 2.1]{AlbertsAveragedInputTauberianTheorems} allows us to make the $\eu{P}$-dependence in the big-O term explicit in the form of \cref{prop:quantitative equidistribution for abelian extensions}.

\begin{proof}[Proof of \cref{prop:quantitative equidistribution for abelian extensions}]
	By \cref{lem:counting series as a finite sum of euler products abelian case},
	\[
	\mc{N}_{k,G,H}^{\eu{P}}(s)
	=
	\sum_{N \le G}
	\frac{\mu(N,G)}{\#{\rm Obs}(k,N,S)}
	\sum_{\beta \in {\rm Obs}(k,N,S)}
	{\rm W}_{k,N,H|_N}^{\beta,\eu{P}}(s).
	\]
	Since there are only finitely many pairs $(N,\beta)$, it suffices to prove a uniform power-saving estimate for each $W_{k,N,H|_N}^{\beta,\eu{P}}(s)$ individually.
	In fact, we may reduce to the case that $N = G$ and fix a parameter $\beta \in {\rm Obs}(k,G,S)$.
	
	Let ${\rm N}(X)$ denote the summatory function of the coefficients of ${\rm W}_{k,G,H}^{\beta,\eu{P}}(s)$ and $\widehat{{\rm N}}(X)$ denote the summatory function of the coefficients of ${\rm W}_{k,G,H}^{0,\eu{P}}(s)$.
	Notice for all $X \ge 10$ that $|{\rm N}(X)| \le \widehat{{\rm N}}(X)$ and $\widehat{{\rm N}}(X)$ is nondecreasing.
	Following the notation in Alberts \cite[\S2.1]{AlbertsAveragedInputTauberianTheorems}, we let
	\[
	L(s,{\rm N}) := W_{k,G,H}^{\beta,\eu{P}}(s), \qquad 
	L(s,\widehat{{\rm N}}) := W_{k,G,H}^{0,\eu{P}}(s)
	\]
	denote the Dirichlet series associated to our summatory functions.
	Notice by \cref{lem:meromorphic continuation Wright and counting series abelian}(i) that both $L$-series converge absolutely for $\Re(s) > 1/a(G,w)$ and admit meromorphic continuations to the half-plane $\Re(s) > 1/(a(G,w)+1)$, and its only possible pole is at $s = 1/a(G,w)$ with order $\le b(k,G,w)$.

	Let $({\rm N}',\beta') \in \{({\rm N},\beta),(\widehat{{\rm N}},0)\}$ and $\beta' \in \{\beta,0\}$.
	In order to apply \cite[Thm.\ 2.1]{AlbertsAveragedInputTauberianTheorems}, we must provide bounds on the \textit{average twisted value}
	\[
	\int_{10}^{T} L(\sigma+it,{\rm N}') Z^{it}\, dt
	\]
	where $T \ge 10, Z\ge 10$ are parameters and $\sigma,t \in \R$ with
	\[
	\frac{1}{a(G,w)+1/2} \le \sigma \le \frac{1}{a(G,w)} + 1.
	\]
	We trivially bound the average twisted value via
	\[
	\left|\int_{10}^{T} L(\sigma+it,{\rm N}) Z^{it}\,dt\right| \le 
	T \max_{10 \le t \le T} |L(\sigma+it,{\rm N})|.
	\]
	As in the proof of \cite[Thm.\ 8.1]{AlbertsPowerSavingsAbelianExtensions}, there exists $\xi > 0$ (depending only on $k,G,H$) and a constant $Q_{\eu{P}} > 0$ (depending only on $k,G,H,\eu{P}$) such that for all $\varepsilon>0$ and $10 \le t \le T$ one has that
	\[
	|L(\sigma+it,{\rm N}')| \ll_{\varepsilon} Q_{\eu{P}} \cdot T^{\xi+\varepsilon}.
	\]

	Tracing through the proof of \cite[Thm.\ 8.1]{AlbertsPowerSavingsAbelianExtensions}, one finds that $Q_{\eu{P}}$ arises from a triangle inequality bound on the numerator of the local ramification factors $f_{\mfp}^{\beta'}(s)$ for $\mfp \in \eu{P}$ and $s = \sigma+it$ (see \cref{def:abelian ramification factors f mfp}).
	Because $\eu{P}$ avoids $S \supset {\rm Exc}(H)$, each modified local factor is a finite Dirichlet polynomial in $(\NN\mfp)^{-s}$ that does \textit{not} depend on $\eu{P}$.
	Thus, there exists $A>0$ depending only on $k,G,H$ such that by writing $\eu{P} = \{\mfp_1,\ldots,\mfp_m\}$ one has that
	\[
	Q_{\eu{P}} \ll_{k,G,H} (\NN\mfp_1 \dotsi \NN\mfp_m)^{A}.
	\]

	Thus, we have verified the hypotheses of \cite[\S2.1]{AlbertsAveragedInputTauberianTheorems} for the summatory functions ${\rm N}(X),\widehat{{\rm N}}(X)$ and may therefore apply \cite[Thm.\ 2.1]{AlbertsAveragedInputTauberianTheorems} with parameters $\eta = \widetilde{\eta} := \xi+1+\varepsilon$ (depending only on $k,G,H,\varepsilon$) and $\beta := 0$ in the notation of \textit{loc.\ cit.}
	Setting $T$ to be a fixed power of $X$ yields a power-saving estimate for ${\rm N}(X)$ with implied constant of the form
	\[
	O_{k,G,H,\varepsilon}\left(
	Q_{\eu{P}} + \max_{|t| \le 10} \left|\frac{L(\sigma+it,{\rm N})}{\sigma+it}\right| +
	\max_{|t| \le 10} \left|\frac{L(\sigma+it,\widehat{{\rm N}})}{\sigma+it}\right|
	\right) =
	O_{k,G,H,\varepsilon}\left(
	Q_{\eu{P}}
	\right).
	\]
	This finishes the proof.
\end{proof}

\subsection{Proof of \cref{thm:Erdos Kac abelian extensions}}
\label{subsec:proof of Erdos Kac abelian extensions}

\begin{proof}[Proof of \cref{thm:Erdos Kac abelian extensions}]

	Let $k,G,w,H,S$ be as in \cref{notation:Erdos Kac abelian extensions}.
	We are able to apply \cref{thm:EK criterion for ideals} by breaking the checks into several steps.
	
	\textit{Step 1. Recognize an ordered sequence of ideals (\cref{def:ordered sequence of ideals in Ok}).}
	We define the indexed set
	\[
	\msF_{k,G} := \{\disc(\varphi) : \varphi \in \Sur(\Gamma_k,G)\}.
	\]
	Then our height function $H:\Hom(\Gamma_k,G) \to \R_{> 0}$ gives a function on $\msF_{k,G}$ that is Northcott by Hermite's theorem, hence $(\msF_{k,G},H)$ is an ordered sequence of ideals in $\O_k$.
	
	\textit{Step 2. Define a Tauberian mixture (\cref{def:Tauberian mixture for ideals}).}
		We define the Tauberian pre-mixture $\mc{T}_{k,G,H}$ via the data:
		\begin{itemize}
			\item $\Lambda(\mc{T}_{k,G,H}) := \coprod_{N \le G} {\rm Obs}(k,N,S)$;
			\item $s_0(\mc{T}_{k,G,H}) := 1/a(G,w)$;
			\item $b(\mc{T}_{k,G,H}) := b(k,G,w)$;
			\item for each $(N,\beta) \in \Lambda(\mc{T}_{k,G,H})$,
			\[
			D_{(N,\beta)}(\mc{T}_{k,G,H};s) := \frac{\mu(N,G)}{\#{\rm Obs}(k,N,S)} {\rm W}_{k,N,H|_N}^{\beta}(s).
			\]
		\end{itemize}
		By \cref{lem:meromorphic continuation Wright and counting series abelian}, $\mc{T}_{k,G,H}$ is a Tauberian mixture.

	\textit{Step 3. Define a Bernoulli process (\cref{def:Bernoulli process for ideals}).}
	We define a Bernoulli process $\mc{B}_{k,G,H}$ over $k$ given by the following data:
	\begin{itemize}
		\item $S(\mc{B}_{k,G,H}) := \{\mfp \subset \O_k \text{ prime} : \mfp \in S\}$;
		\item $\mc{T}(\mc{B}_{k,G,H}) := \mc{T}_{k,G,H}$;
		\item $X_0 \ge 10$ is chosen so that for all $X\ge X_0$ one has that
		\[
		\underset{s = 1/a(G,w)}{\Res} \left(
			\mc{N}_{k,G,H}(s) X^s/s
		\right) \ne 0.
		\]
		\item For every prime $\mfp \notin S$, define $\boldsymbol{f}_{\mfp}(\mc{B}_{k,G,H};s) := \Big(f_{\mfp}^{N,\beta}(s)\Big)_{\substack{N \le G \\ \beta \in {\rm Obs}(k,N,S)}}$ where $f_{\mfp}^{N,\beta}(s)$ are the local ramification factors defined in \cref{def:abelian ramification factors f mfp}.
	\end{itemize}

	\textit{Step 4. Check that $\mc{B}_{k,G,H}$ is well-mixed (\cref{def:well-mixed Bernoulli process}).}
	\begin{itemize}
			\item[] \textit{Mixture uniformity.} 
			Let $s_0 = s_0(\mc{T}_{k,G,H}) = 1/a(G,w)$.
			We observe for all $(N,\beta) \in \Lambda$ that $f_{\mfp}^{N,\beta}(s)$ is a ratio of Dirichlet polynomials (see \cref{def:abelian ramification factors f mfp}) whose nonconstant monomial terms are $O((\NN\mfp)^{-a(G,w)s})$ near $s = 1/a(G,w)$, the denominator of $f_{\mfp}^{N,\beta}(s)$ evaluated at $s_0$ is uniformly bounded away from zero as $\mfp$ varies by the proof of \cref{lem: f mfp (s) are holomorphic}, and therefore by the chain rule one has for all $r \in \Z_{\ge 0}$ that
			\[
			\frac{d^r}{ds^r} f_{\mfp}^{\beta}(s)\bigg|_{s = 1/a(G,w)} \ll_{k,G,w,r} \frac{(\log \NN \mfp)^r}{\NN\mfp}.
			\]
			\item[] \textit{Leading order conditional probabilities are similar.}
			The set $\Lambda^+(\mc{T}_{k,G,H})$ consists of the pairs $(N,\beta)$ such that $D_{(N,\beta)}(\mc{T}_{k,G,H};s)$ has a pole at $s = 1/a(G,w)$ of largest possible order.
			Equivalently, $\mu(N,G)\neq 0$ and $W_{k,N,H|_N}^{\beta}(s)$ has a pole at $s = 1/a(G,w)$ of largest possible order.
			By \cref{lem:meromorphic continuation Wright and counting series abelian}, an upper bound on the order of this pole is given by $b(k,N,w|_N) \le b(k,G,w)$, hence it is necessary for the subgroup $N \le G$ to satisfy $a(N,w|_N) = a(G,w)$ and $b(k,N,w|_N) = b(k,G,w)$ in order to contribute to the set $\Lambda^+(\mc{T}_{k,G,H})$.
			By \cref{lem:meromorphic continuation Wright and counting series abelian}(iii), we may enlarge our finite set of places $S$ so that for all primes $\mfp \notin S$ and $(N,\beta) \in \Lambda^+(\mc{T}_{k,G,H})$ we have that
		\begin{equation}
		\label{eq:abelian EK leading order conditional probabilities}
		f_{\mfp}^{N,\beta}(s_0)  = \frac{\nu_{k,G,w,\Frob_{\mfp}}}{\NN\mfp} 
		+ O_{G}\left(\frac{1}{(\NN\mfp)^{1+1/a(G,w)}}\right)
		\end{equation}
		where $\nu_{k,G,w,\sigma}$ is the number of minimal weight elements of $G\backslash\{1\}$ invariant under $g \mapsto g^{{\rm cyc}^{-1}(\sigma)}$ where ${\rm cyc}^{-1}:\Gal(k(\zeta_{\exp(G)})/k)\to (\Z/\exp(G)\Z)^\times$ is the anticyclotomic character.
		This proves the desired claim that \cref{def:well-mixed Bernoulli process}(ii) holds for $\mc{B}_{k,G,H}$ with $\alpha = 1/a(G,w)$.
	\end{itemize}

	\textit{Step 5. Check reasonability (\cref{def:reasonable sequence for ideals}).}
		We have to check two things.
		\begin{itemize}
			\item[] \textit{Large prime tails are sparse.} This follows from \cref{lem:restricted discriminants have sparse tails}.
			\item[] \textit{Approximation by $\mc{B}_{k,G,H}$.} This follows from \cref{prop:quantitative equidistribution for abelian extensions}.
	\end{itemize}

	Steps 1--5 prove that $(\msF_{k,G},H)$ is a reasonable sequence of ideals with respect to a well-mixed Bernoulli process $\mc{B}_{k,G,H}$; therefore, in order to apply all parts of \cref{thm:EK criterion for ideals} it suffices to compute the normalized algebraic mean and variance of the number of ramified primes along the sequence and show that the algebraic mean converges at a sufficient rate.
	
	\textit{Step 6. The mean and variance are of order $b \log\log X$.}
	In view of \cref{eq:abelian EK leading order conditional probabilities}, it suffices to understand the average value of $\nu_{k,G,w,\Frob_{\mfp}}$ as $\mfp$ varies over primes of $k$.
	The expected value of $\nu_{k,G,w,\sigma}$ where $\sigma \in \Gal(k(\zeta_{\exp(G)})/k)$ is uniformly chosen is equal to the average number of fixed points for the set of minimal weight conjugacy classes of $G$ under the anticyclotomic action over $k$, which by Burnside's lemma is equal to the number of orbits of this action, and by \cref{rem:number of sectors is Malle b} this is equal to $b(k,G,w)$.
		Thus, we have by the Chebotarev density theorem and Mertens' theorem over $k$ that
		\begin{equation}
		\label{eq:rate of convergence of algebraic mean EK abelian}
		\sum_{\substack{\NN \mfp \leq X \\ \mfp \notin S}} \frac{\nu_{k,G,w,\Frob_\mfp}}{\NN\mfp} =
		b(k,G,w) \log\log X + O_{k,G,w}(1)
		\end{equation}
		as $X \to \infty$.
	This implies that we have
	\[
	\mu = V = \sum_{(N,\beta) \in \Lambda^+} w_{(N,\beta)}(\infty)
	\lim_{X \to \infty} 
	\frac{1}{\log\log X}
	\sum_{\substack{\mfp \notin S \\ \NN\mfp \le X}} f_{\mfp}^{N,\beta}(s_0) = 
	b(k,G,w) \sum_{\lambda \in \Lambda^+} w_{\lambda}(\infty) = b(k,G,w).
	\]
	Moreover, \cref{eq:abelian EK leading order conditional probabilities,eq:rate of convergence of algebraic mean EK abelian} show that the rate of convergence of the above expression is $O(\frac{1}{\log\log X}) = o((\log\log X)^{-1/2})$, therefore \cref{thm:EK criterion for ideals}(iii) applies. 

	This finishes the proof.
\end{proof}

\section{Probability theory II: a method of cumulants}
\label{sec:probability theory II}

In this section, we prove \cref{thm:EK criterion for ideals} using the method of cumulants.

\begin{notation}
	\label{notation:probability theory II}
	Throughout this section, we let $k$ be a number field and fix a well-mixed Bernoulli process
	\[
	\mc{B} =
	\mc{B}\Big(S,\mc{T}\Big(\Lambda,s_0,b,(D_\lambda(s))_{\lambda \in \Lambda}\Big),X_0,(\boldsymbol{f}_{\mfp}(s))_{\mfp \notin S}\Big)
	\]
	over $k$ with parameter $\alpha > 0$ appearing in \cref{def:well-mixed Bernoulli process}(ii) for $\mc{B}$.
	We fix an ordered sequence of ideals $(\msF,H)$ in $\O_k$ that is reasonable with respect to $\mc{B}$ with sequence of positive constants $(A_m,\delta_m)_{m\in\N}$ appearing in the reasonability condition for $(\msF,H)$ with respect to $\mc{B}$ (see \cref{def:reasonable sequence for ideals}(ii)).
	Finally, we fix a strongly additive function $\omega : \mf{N}_k \to \R$ with $\Vert \omega \Vert_{\rm sup} < \infty$.

	We write $\boldsymbol{f}_{\mfp}(s)=(f_{\mfp}^{\lambda}(s))_{\lambda\in\Lambda}$ and $\Lambda^+=\Lambda^+(\mc{T})$.
	For all $X\ge X_0$ and $\lambda \in \Lambda$, we write
	\[
	w_{\lambda}(X):=
	\frac{\underset{s=s_0}{\Res}\big(D_\lambda(s)X^s/s\big)}
	{\underset{s=s_0}{\Res}\big(\sum_{\lambda'\in\Lambda}D_{\lambda'}(s)X^s/s\big)}.
	\]
	In particular, $\sum_{\lambda\in\Lambda}w_{\lambda}(X)=1$ and by \cref{rem:residue is a polynomial in X} the limits $w_\lambda(\infty):=\lim\limits_{X\to\infty}w_\lambda(X)$ exist.
\end{notation}

\begin{framed}
	In this section, all implied constants are allowed to depend on the fixed number field $k$, well-mixed Bernoulli process $\mc{B}$ with parameter $\alpha > 0$ appearing in \cref{def:well-mixed Bernoulli process}(ii) for $\mc{B}$, ordered sequence of ideals $(\msF,H)$ in $\O_k$ that is reasonable with respect to $\mc{B}$ with sequence of positive constants $(A_m,\delta_m)_{m \in \N}$, and strongly additive function $\omega : \mf{N}_k \to \R$ with $\Vert \omega \Vert_{\rm sup} < \infty$.
\end{framed}

\subsection{Algebraic cumulants}

Given a finite set $J$, let $\Pi_J$ be the lattice of set partitions of $J$ ordered by refinement $\preceq$, i.e.\ two partitions $\sigma,\tau$ satisfy $\sigma \preceq \tau$ if and only if $\sigma$ is a refinement of $\tau$.
We write $\widehat{1} \in \Pi_J$ for the maximal element, i.e.\ $\widehat{1} := \{J\}$.
If $m\in\N$ and $\eu{P}$ is an $m$-tuple (of elements of some set) with indexing set $J=\{1,\ldots,m\}$, we write $\Pi(\eu{P})$ for $\Pi_J$, and interpret a subset $B\subset J$ (also called a \textit{block}) as the indexing set of a subtuple of $\eu{P}$.
We write $\mu(\pi) := \mu_J(\pi,\widehat{1})$ for the Möbius function of $(\Pi_J,\preceq)$ evaluated on the interval from $\pi$ to $\widehat{1}$.

\begin{definition}
	For every finite tuple $\eu{P}$ of primes of $k$ avoiding $S$ and every $X\ge X_0$, we define the \textit{algebraic $\eu{P}$-cumulant}
	\[
	\widehat{\kappa}_{\mc{B},X}(\eu{P})
	:=
	\sum_{\pi\in\Pi(\eu{P})}\mu(\pi)\prod_{B\in\pi}\widehat{M}_{\mc{B},X}(B),
	\]
	and similarly for all $\lambda\in\Lambda$ we define the \textit{$\lambda$-conditional algebraic $\eu{P}$-cumulant}
	\[
	\widehat{\kappa}_{\mc{B},X}(\eu{P}\mid\lambda)
	:=
	\sum_{\pi\in\Pi(\eu{P})}\mu(\pi)\prod_{B\in\pi}\widehat{M}_{\mc{B},X}(B\mid\lambda).
	\]
\end{definition}

For example, one has that
\[
\widehat{\kappa}_{\mc{B},X}(\mfp)=\widehat{M}_{\mc{B},X}(\mfp),
\]
\[
\widehat{\kappa}_{\mc{B},X}(\mfp_1,\mfp_2)=
\begin{cases}
\widehat{M}_{\mc{B},X}(\mfp_1,\mfp_2)-\widehat{M}_{\mc{B},X}(\mfp_1)\widehat{M}_{\mc{B},X}(\mfp_2) & \mfp_1\ne\mfp_2,\\
\widehat{M}_{\mc{B},X}(\mfp)-\widehat{M}_{\mc{B},X}(\mfp)^2 & \mfp_1=\mfp_2=\mfp.
\end{cases}
\]

\begin{definition}
Given $X\ge X_0$, a finite set $J$, and functions $F_j:\Lambda\to\C$ indexed by $j \in J$, we define the \textit{joint cumulant}
\[
\kappa_{\Lambda,\mc{T},X}((F_j:j\in J))
:=
\sum_{\pi\in\Pi_J}\mu(\pi)
\prod_{B\in\pi}
\sum_{\lambda\in\Lambda}w_\lambda(X)\prod_{j\in B}F_j(\lambda).
\]
\end{definition}

\begin{lemma}[Law of total cumulance]
	\label{lem:law of total cumulance}
	For every finite tuple $\eu{P}$ of primes of $k$ avoiding $S$ and $X\ge X_0$ one has that
	\[
	\widehat{\kappa}_{\mc{B},X}(\eu{P})
	=
	\sum_{\pi\in\Pi(\eu{P})}
	\kappa_{\Lambda,\mc{T},X}((\widehat{\kappa}_{\mc{B},X}(B\mid\lambda):B\in\pi)).
	\]
\end{lemma}

\begin{proof}
	This is the usual law of total cumulance applied to the algebraic probability space $\mc{A}_0(\mc{T})$ and the subspace $\C^\Lambda$.
	Indeed, the conditional expectation from \cref{construction:algebraic probability space associated to Tauberian mixture} is $\C$-linear and unital, and by definition
	\[
	\widehat{\mathbb{E}}_{\mc{T},X}
	=
	\widehat{\mathbb{E}}_{\Lambda,\mc{T},X}\circ
	\widehat{\mathbb{E}}_{\mc{T},X}[-\mid\Lambda].
	\]
	This finishes the proof.
\end{proof}

It immediately follows from the law of total cumulance that we may write
\begin{equation}
\label{eq:law of total cumulance}
\widehat{\kappa}_{\mc{B},X}(\eu{P})
=
\sum_{\lambda\in\Lambda}w_\lambda(X)\widehat{\kappa}_{\mc{B},X}(\eu{P}\mid\lambda)
+
\sum_{\substack{\pi\in\Pi(\eu{P})\\ |\pi|\ge 2}}
\kappa_{\Lambda,\mc{T},X}((\widehat{\kappa}_{\mc{B},X}(B\mid\lambda):B\in\pi)).
\end{equation}
We refer to the cumulants in the first sum as \textit{pure cumulants}, and the cumulants in the second sum as \textit{mixed cumulants}.

\subsection{Bounding algebraic cumulants}
\label{subsec:bounding algebraic cumulants}

Let $\lambda\in\Lambda$ and $X\ge X_0$.
By \cref{rem:convention for measure zero conditional expectation}, we may assume that $D_{\lambda}(s)$ has a pole at $s=s_0$, in which case we introduce the shorthand notation
\[
\mc{D}_{s,X}^{\lambda}[g(s)]
:=
\frac{\underset{s=s_0}{\Res}\big(g(s)D_\lambda(s)X^s/s\big)}
{\underset{s=s_0}{\Res}\big(D_\lambda(s)X^s/s\big)}.
\]
Since $D_\lambda(s)$ has a pole of order at most $b$ at $s=s_0$, it follows from the Cauchy residue theorem that $\mc{D}_{s,X}^{\lambda}$ is a differential operator of order at most $b-1$ of the form
\[
\mc{D}_{s,X}^{\lambda}
=
\left(1+\sum_{r=1}^{b-1}\frac{1}{r!}\frac{P_\lambda^{(r)}(\log X)}{P_\lambda(\log X)}\frac{\partial^r}{\partial s^r}\right)\Bigg|_{s=s_0}
\]
where $P_\lambda(t) \in \C[t]$ is a nonzero polynomial of degree at most $b-1$.

For a tuple $\eu{P}$ of primes avoiding $S$, we define the \textit{$\lambda$-conditional cumulant generating function}
\begin{equation}
\label{eq:cumulant generating function Bernoulli version}
\widehat{K}_X^\lambda(\eu{P};\mathbf{t})
:=
\log \mc{D}_{s,X}^{\lambda}\left[
\prodflat_{\mfp\in\eu{P}}
\bigl(1+(e^{t_\mfp}-1)f_\mfp^\lambda(s)\bigr)
\right],
\end{equation}
where $\mathbf{t}=(t_\mfp)_{\mfp\in\eu{P}}^{\flat}$.
Here we have implemented an algebraic probability version of the fact from classical probability theory that if $Z$ is a Bernoulli random variable, then
\[
e^{tZ} = 1 + (e^t-1)Z,
\]
since $Z^i = Z$ for all $i \in \Z_{\ge 1}$.
If $m_\mfp \ge 1$ is the multiplicity of $\mfp$ in $\eu{P}$ and $m$ is the index set cardinality of $\eu{P}$, then the standard Taylor expansion for joint cumulants gives
\begin{equation}
\label{eq:cumulant Taylor series coefficients}
\widehat{\kappa}_{\mc{B},X}(\eu{P}\mid\lambda)
=
\frac{\partial^m}{\sideset{}{^{\flat}_{\mfp\in\eu{P}}}{\prod}\partial t_\mfp^{m_\mfp}}
\widehat{K}_X^\lambda(\eu{P};\mathbf{t})\bigg|_{\mathbf{t}=0}.
\end{equation}

\begin{lemma}[Pure cumulant bound]
	\label{lem:pure cumulant bound}
	Let $X\ge 10$, $m\in \Z_{\ge 2}$, and $\eu{P}$ an $m$-tuple of primes of $k$ avoiding $S$.
	If $\eu{P}$ contains at least two distinct primes and $\NN\mfp\le X$ for every $\mfp\in\eu{P}$, then for all $\lambda\in\Lambda$ one has that
	\[
	\widehat{\kappa}_{\mc{B},X}(\eu{P}\mid\lambda)
	\ll_m
	\prodflat_{\mfp\in\eu{P}}
	\frac{1}{\NN\mfp}\frac{\log\NN\mfp}{\log X}.
	\]
\end{lemma}

\begin{proof}
	Let $\lambda \in \Lambda$.
	Because $\eu{P}$ contains at least two distinct primes, \cref{eq:cumulant Taylor series coefficients} says that $\widehat{\kappa}_{\mc{B},X}(\eu{P} \mid \lambda)$ is a \textit{mixed} Taylor coefficient of \cref{eq:cumulant generating function Bernoulli version}.
	In particular, the generating function
	\[
	\log \prodflat_{\mfp \in \eu{P}} (1 + (e^{t_{\mfp}}-1) f_{\mfp}^{\lambda}(s_0))
	\]
	has vanishing mixed Taylor coefficients, so shifting \cref{eq:cumulant generating function Bernoulli version} by this generating function yields
	\[
	\widehat{\kappa}_{\mc{B},X}(\eu{P} \mid \lambda) = 
	\frac{\partial^m}{\sideset{}{^{\flat}_{\mfp\in\eu{P}}}{\prod}\partial t_\mfp^{m_\mfp}}
	\log\left(
		1 + (\mc{D}_{s,X}^{\lambda}-1)\left[
			\prodflat_{\mfp \in \eu{P}}
			\left(
			1 +
			\frac{(e^{t_{\mfp}}-1)(f_{\mfp}^{\lambda}(s)-f_{\mfp}^{\lambda}(s_0))}
			{1 + (e^{t_{\mfp}}-1)f_{\mfp}^{\lambda}(s_0)}
			\right)
		\right]
	\right)
	\Bigg|_{\mathbf{t}=0}.
	\]
	We write the Taylor expansion of $(\mc{D}_{s,X}^{\lambda}-1)[\sideset{}{^{\flat}_{\mfp \in \eu{P}}}{\prod}\dotsi]$ centered at $\boldsymbol{t} = 0$ as
	\[
	\sum_{\substack{\boldsymbol{i} = (i_{\mfp})_{\mfp \in \eu{P}}^{\flat} \\ i_{\mfp} \ge 0}}
	\tau(\boldsymbol{i}) \prodflat_{\mfp \in \eu{P}} t_{\mfp}^{i_{\mfp}}, \qquad \tau(\boldsymbol{i}) \in \C,
	\]
	so that
	\begin{equation}
		\label{eq:pure cumulant log Taylor series expression}
		\widehat{\kappa}_{\mc{B},X}(\eu{P} \mid \lambda) = 
	\frac{\partial^m}{\sideset{}{^{\flat}_{\mfp\in\eu{P}}}{\prod}\partial t_\mfp^{m_\mfp}}
	\log\left(
		1 + \sum_{\substack{\boldsymbol{i} = (i_{\mfp})_{\mfp \in \eu{P}}^{\flat} \\ i_{\mfp} \ge 0}}
	\tau(\boldsymbol{i}) \prodflat_{\mfp \in \eu{P}} t_{\mfp}^{i_{\mfp}}
	\right)
	\Bigg|_{\mathbf{t}=0}.
	\end{equation}

	The differential operator $\mc{D}_{s,X}^{\lambda}-1$ acts on the bracketed expression $[\sideset{}{^{\flat}_{\mfp \in \eu{P}}}{\prod} \dotsi]$, and when we expand $(\mc{D}_{s,X}^{\lambda}-1)[\sideset{}{^{\flat}_{\mfp \in \eu{P}}}{\prod} \dotsi]$ using the Leibniz rule we obtain a sum with $O_m(1)$ summands, depending on $\boldsymbol{t}$, whose Taylor coefficients centered at $\boldsymbol{t} = 0$ are polynomials in the derivatives of $f_{\mfp}^{\lambda}(s)$ evaluated at $s=s_0$; we call these summands that are monomials in $(t_{\mfp})_{\mfp \in \eu{P}}^\flat$ the \textit{terms} of $(\mc{D}_{s,X}^{\lambda}-1)[\sideset{}{^{\flat}_{\mfp \in \eu{P}}}{\prod}\dotsi]$.

	For all $1 \le r < b$, \cref{def:well-mixed Bernoulli process}(i) for $\mc{B}$ implies that
	\begin{equation}
	\label{eq:derivative bound on f mfp s in cumulant proof}
	\frac{1}{r!}\frac{P_\lambda^{(r)}(\log X)}{P_\lambda(\log X)}
	\frac{d^r}{ds^r}f_\mfp^\lambda(s)\bigg|_{s=s_0}
	\ll_r
	\frac{1}{\NN\mfp}\left(\frac{\log\NN\mfp}{\log X}\right)^r.
	\end{equation}
	Notice that a term has a factor of $t_{\mfp}$ if and only if it contains a \textit{positive} derivative of $f_{\mfp}^{\lambda}(s)$ evaluated at $s = s_0$.
	It follows that if a term of $(\mc{D}_{s,X}^{\lambda}-1)[\sideset{}{^{\flat}_{\mfp \in \eu{P}}}{\prod} \dotsi]$ contains a factor of $t_\mfp$, then our upper bound for the term using the triangle inequality and \cref{eq:derivative bound on f mfp s in cumulant proof} is smaller by a factor of $\log \NN\mfp / \log X$.
	Since an order $m$ Taylor coefficient is the sum of $O_m(1)$ many terms, it follows for all indices $\boldsymbol{i}$ that
	\[
	\tau(\boldsymbol{i})
	\ll_m 
	\prodflat_{\{\mfp \in \eu{P} : i_{\mfp} > 0\}}
	\frac{1}{\NN \mfp}
	\frac{\log \NN\mfp}{\log X}.
	\]
	In view of \cref{eq:pure cumulant log Taylor series expression}, $\widehat{\kappa}_{\mc{B},X}(\eu{P} \mid \lambda)$ is a sum of $O_m(1)$ products of the Taylor coefficients $\tau(\boldsymbol{i})$ with $\sum_{\mfp\in\eu{P}}^{\flat} i_{\mfp} \le m$, times order $\le m$ Taylor coefficients of $\log$ centered at $1$.
	This is enough to deduce the desired bound, since $m_{\mfp} > 0$ for all $\mfp \in \eu{P}$.
\end{proof}

\begin{lemma}[Mixed cumulant bound]
	\label{lem:mixed cumulant bound}
	Let $X\ge 10$, $m\in \Z_{\ge 2}$, $\eu{P}$ an $m$-tuple of primes of $k$ avoiding $S$, and $\pi\in\Pi(\eu{P})$ with $|\pi|\ge 2$.
	If $\NN\mfp\le X$ for every $\mfp\in\eu{P}$, then
	\[
	\kappa_{\Lambda,\mc{T},X}((\widehat{\kappa}_{\mc{B},X}(B\mid\lambda):B\in\pi))
	\ll_m
	\frac{1}{\log X}
	\prodflat_{\mfp\in\eu{P}}
	\frac{1}{\NN\mfp} +
	\prodflat_{\mfp\in\eu{P}}
	\left(
	\frac{1}{\NN\mfp}\frac{\log\NN\mfp}{\log X}
	+
	\frac{1}{(\NN\mfp)^{\min\{2,1+\alpha\}}}
	\right).
	\]
\end{lemma}

\begin{proof}
	Since $|\pi|\ge 2$, the outer cumulant $\kappa_{\Lambda,\mc{T},X}$ is unchanged if each argument $\widehat{\kappa}_{\mc{B},X}(B\mid\lambda), B \in \pi$ is shifted by a constant that is \textit{deterministic}, i.e.\ independent of $\lambda \in \Lambda$.
	For all $B \in \pi$ we choose
	\[
	\widehat{\Delta}_{\mc{B},X}(B\mid\lambda)
	:=
	\widehat{\kappa}_{\mc{B},X}(B\mid\lambda)
	-
	\sum_{\lambda'\in\Lambda}w_{\lambda'}(X)\widehat{\kappa}_{\mc{B},X}(B\mid\lambda')
	\]
	as a deterministic shift in the $B$-th entry of the outer cumulant, so that we have
	\[
	\kappa_{\Lambda,\mc{T},X}((\widehat{\kappa}_{\mc{B},X}(B\mid\lambda):B\in\pi))
	=
	\kappa_{\Lambda,\mc{T},X}((\widehat{\Delta}_{\mc{B},X}(B\mid\lambda):B\in\pi)).
	\]
	We unpack the definition of this cumulant as
	\[
	\kappa_{\Lambda,\mc{T},X}((\widehat{\Delta}_{\mc{B},X}(B\mid\lambda):B\in\pi)) = 
	\sum_{\sigma \in \Pi(\pi)} \mu(\sigma) \prod_{U \in \sigma} 
	\sum_{\lambda \in \Lambda} w_{\lambda}(X) \prod_{B \in U} 
	\widehat{\Delta}_{\mc{B},X}(B\mid\lambda).
	\]
	It suffices to prove that, for all $U \subset \pi$, one has that
	\[
	\max_{\lambda \in \Lambda}
	\left|
		w_{\lambda}(X) \prod_{B \in U} \widehat{\Delta}_{\mc{B},X}(B\mid\lambda)\right|
	\ll_m
	\frac{1}{\log X}
	\prod_{B \in U} \prodflat_{\mfp \in B}
	\frac{1}{\NN\mfp} +
	\prod_{B \in U}
	\prodflat_{\mfp \in B}
	\left(
	\frac{1}{\NN\mfp}\frac{\log\NN\mfp}{\log X}
	+ \frac{1}{(\NN\mfp)^{\min\{2,1+\alpha\}}}
	\right).
	\]
	We split this claim into two cases: the case that $\lambda \notin \Lambda^+$ and the case that $\lambda \in \Lambda^+$.

	We first consider the case that $\lambda \notin \Lambda^+$.
	Then $w_{\lambda}(X) = O(1/\log X)$ by \cref{rem:residue is a polynomial in X}, and together with \cref{def:well-mixed Bernoulli process}(i) for $\mc{B}$ this implies for all $U \subset \pi$ that
	\[
	\max_{\lambda \notin \Lambda^+}
	\left|
		w_{\lambda}(X) \prod_{B \in U} \widehat{\Delta}_{\mc{B},X}(B\mid\lambda)
	\right|
	\ll_m 
	\frac{1}{\log X} 
	\prod_{B \in U}
	\prodflat_{\mfp \in B}
	\frac{1}{\NN\mfp}.
	\]

	We now consider the case that $\lambda \in \Lambda^+$.
	If a block $B\in\pi$ contains at least two distinct primes, then \cref{lem:pure cumulant bound} gives
	\[
	\max_{\lambda\in\Lambda^+}|\widehat{\Delta}_{\mc{B},X}(B\mid\lambda)|
	\ll_m
	\prodflat_{\mfp\in B}
	\frac{1}{\NN\mfp}\frac{\log\NN\mfp}{\log X}.
	\]
	It remains to handle a block $B$ all of whose entries are equal to a single prime.
	For such a block, we use the well-known fact from classical probability theory that if $Z$ is a (virtual) Bernoulli random variable with $z := \mathbb{E}[Z]$, then the order $m$ cumulant $\kappa(Z,\ldots,Z)$ is a degree $m$ polynomial in $z$ of the form $z(1+O_m(z))$.
	Therefore, in the case that $B = (\mfp,\ldots,\mfp)$ consists of a single prime $\mfp$, we have for all $\lambda \in \Lambda$ that
	\[
	\widehat{\kappa}_{\mc{B},X}(B\mid\lambda)
	=
	\widehat{z} (1+O_m(\widehat{z})),
	\]
	where
	\[
	\widehat{z} := \widehat{\mathbb{E}}_{\mc{T},X}[\boldsymbol{f}_{\mfp}(s_0) \mid \lambda] = 
	f_\mfp^\lambda(s_0) + O_m\left(
		\frac{1}{\NN\mfp} \frac{\log \NN\mfp}{\log X}
	\right).
	\]
	Substituting this into the above formula yields
	\[
	\widehat{\kappa}_{\mc{B},X}(B\mid\lambda) = 
	f_\mfp^\lambda(s_0)
	+
	O_m\left(\frac{1}{\NN\mfp}\frac{\log\NN\mfp}{\log X}+ \frac{1}{(\NN\mfp)^2}\right).
	\]
	Using \cref{def:well-mixed Bernoulli process} for $\mc{B}$, we find for all $\lambda \in \Lambda$ that
	\begin{align*}
		\widehat{\Delta}_{\mc{B},X}(B\mid\lambda) &= \widehat{\kappa}_{\mc{B},X}(B\mid \lambda) - \sum_{\lambda' \in \Lambda} w_{\lambda'}(X) \widehat{\kappa}_{\mc{B},X}(B \mid \lambda')
		\\ &= 
		\sum_{\lambda' \in \Lambda}
		w_{\lambda'}(X) (f_{\mfp}^{\lambda}(s_0) - f_{\mfp}^{\lambda'}(s_0))
		+ O_m\left(\frac{1}{\NN\mfp}\frac{\log\NN\mfp}{\log X} + \frac{1}{(\NN\mfp)^2}\right)
		\\ &=
		O_m\left(
			\frac{1}{\NN\mfp} \frac{\log \NN\mfp}{\log X} + 
			\frac{1}{(\NN\mfp)^{\min\{2,1+\alpha\}}}
		\right).
	\end{align*}
	This yields the desired bound
	\[
	\max_{\lambda\in\Lambda^+}|\widehat{\Delta}_{\mc{B},X}(B\mid\lambda)|
	\ll_m
	\frac{1}{\NN\mfp}\frac{\log\NN\mfp}{\log X}
	+
	\frac{1}{(\NN\mfp)^{\min\{2,1+\alpha\}}},
	\]
	completing the proof of the claim.
\end{proof}

\subsection{Cumulants of a smoothed additive function}

For every finite tuple $\eu{P}$ of primes of $k$ and $X \ge 10$, we define the \textit{expectation} of a function $F : \msF_H(X) \to \C$ to be
\[
\mathbb{E}_{\msF,H,X}[F] := \frac{1}{\#\msF_H(X)} \sum_{\mfn \in \msF_H(X)} F(\mfn).
\]
Thus, we define the \textit{empirical $\eu{P}$-moment} (in contrast with the algebraic $\eu{P}$-moment $\widehat{M}_{\mc{B},X}(\eu{P})$; see \cref{eq:algebraic Sigma moment}) to be
\[
M_{\msF,H,X}(\eu{P})
:= 
\mathbb{E}_{\msF,H,X}\left[ \prod_{\mfp \in \eu{P}}\1_{\mfp \mid \mfn} \right] =
\frac{\#\{\mfn\in\msF_H(X): \mfp\mid\mfn \textnormal{ for every } \mfp\in\eu{P}\}}{\#\msF_H(X)}.
\]
Since our probabilistic method focuses on cumulants rather than moments, we instead study for $\eu{P},X$ as above the \textit{empirical $\eu{P}$-cumulant}
\begin{equation}
\label{eq:empirical Sigma cumulant}
\kappa_{\msF,H,X}(\eu{P}) =
\sum_{\pi\in\Pi(\eu{P})}\mu(\pi)\prod_{B\in\pi}M_{\msF,H,X}(B).
\end{equation}
For every parameter $0<\psi\le 1$ (possibly depending on $X$), we are interested in the \textit{$\psi$-smoothing of $\omega$}, which we define to be the strongly additive function $\omega_{\psi,X}:\mf{N}_k \to \R$ given by
\[
\omega_{\psi,X}(\mfn)
:=
\sum_{\substack{\mfp\notin S,\ \mfp\mid\mfn\\ \NN\mfp\le X^\psi}}
\omega(\mfp).
\]
For every $m \in \N$, the \textit{$m$-th cumulant} of $\omega_{\psi,X}$ on the indexed set $\msF_H(X)$ is given by
\[
\kappa_{\msF,H,X}^{(m)}(\omega_{\psi,X})
:=
\sum_{\substack{\eu{P}=(\mfp_1,\ldots,\mfp_m)\\ \mfp_j\notin S,\ \NN\mfp_j\le X^\psi}}
\omega(\mfp_1)\cdots\omega(\mfp_m)
\kappa_{\msF,H,X}(\eu{P}),
\]
where $\eu{P}$ varies over $m$-tuples of primes of $k$ avoiding $S$ with norm $\le X^{\psi}$.

\begin{lemma}[Approximation of $\psi$-smoothed cumulants]
	\label{lem:approximation of smoothed cumulants}
	Let $m\in \N$ and $X \ge X_0$. For all $0 < \psi \le \frac{1}{m(A_m+1)}(\log\log X)^{-1/2+\delta_m}$, one has that
	\[
	\kappa_{\msF,H,X}^{(m)}(\omega_{\psi,X})
	=
	\sum_{\substack{\eu{P}=(\mfp_1,\ldots,\mfp_m)\\ \mfp_j\notin S,\ \NN\mfp_j\le X^\psi}}
	\omega(\mfp_1)\cdots\omega(\mfp_m)
	\widehat{\kappa}_{\mc{B},X}(\eu{P})
	+O_m(1).
	\]
\end{lemma}

\begin{proof}
	For every $m$-tuple $\eu{P}$ of primes of $k$ avoiding $S$ with norm $\le X^{\psi}$ and every block $B\in\Pi(\eu{P})$, \cref{def:reasonable sequence for ideals}(ii) gives
	\[
	M_{\msF,H,X}(B)
	=
	\widehat{M}_{\mc{B},X}(B)
	+O\left(X^{mA_m\psi - (\log\log X)^{-1/2+\delta_m}}\right).
	\]
	Expanding \cref{eq:empirical Sigma cumulant} with the above estimate for each $M_{\msF,H,X}(B)$ and bounding cross terms using $M_{\msF,H,X}(B) \ll_m 1$ (possibly after enlarging $X_0$) yields
	\[
	\kappa_{\msF,H,X}(\eu{P})
	=
	\widehat{\kappa}_{\mc{B},X}(\eu{P})
	+O_m\left(X^{mA_m\psi - (\log\log X)^{-1/2+\delta_m}}\right)
	\]
	There are $O_m(X^{m\psi})$ many choices for $\eu{P}$ with the above properties, and $\Vert\omega\Vert_{\rm sup}<\infty$, therefore
	\begin{align*}	
	\kappa_{\msF,H,X}^{(m)}(\omega_{\psi,X})
	&=
	\sum_{\substack{\eu{P}=(\mfp_1,\ldots,\mfp_m)\\ \mfp_j\notin S,\ \NN\mfp_j\le X^\psi}}
	\omega(\mfp_1)\cdots\omega(\mfp_m)
	\kappa_{\msF,H,X}(\eu{P})
	\\ &= 
	\sum_{\substack{\eu{P}=(\mfp_1,\ldots,\mfp_m)\\ \mfp_j\notin S,\ \NN\mfp_j\le X^\psi}}
	\omega(\mfp_1)\cdots\omega(\mfp_m)
	\widehat{\kappa}_{\mc{B},X}(\eu{P}) + 
	O_m\left(
		1 + X^{m(A_m+1)\psi - (\log \log X)^{-1/2+\delta_m}}
	\right).
	\end{align*}
	This finishes the proof.
\end{proof}

\begin{lemma}[Off-diagonal tuples are negligible]
	\label{lem:off diagonal tuples are negligible}
	Let $m\in \Z_{\ge 2}$ and $X \ge X_0$.
	Then for all $0 < \psi \le 1$, one has that
	\[
	\sum_{\substack{\eu{P}=(\mfp_1,\ldots,\mfp_m)\\ \mfp_j\notin S,\ \NN\mfp_j\le X^\psi\\ \eu{P}\textnormal{ contains at least two distinct primes}}}
	\left|\omega(\mfp_1)\cdots\omega(\mfp_m)\widehat{\kappa}_{\mc{B},X}(\eu{P})\right|
	\ll_m 1.
	\]
\end{lemma}

\begin{proof}
	Since for all primes $\mfp_1,\ldots,\mfp_m$ one has that $|\omega(\mfp_1) \dotsi \omega(\mfp_m)| \le \Vert \omega\Vert_{\rm sup}^m \ll_m 1$, it suffices to replace $\omega$ by the constant function $1$ in the lemma statement.
	By \cref{eq:law of total cumulance,lem:pure cumulant bound,lem:mixed cumulant bound}, every $m$-tuple $\eu{P}$ of primes of $k$ avoiding $S$ and containing at least two distinct primes satisfies
	\[
	\widehat{\kappa}_{\mc{B},X}(\eu{P})
	\ll_m
	\prodflat_{\mfp\in\eu{P}}
	\left(
	\frac{1}{\NN\mfp}\frac{\log\NN\mfp}{\log X}
	+
	\frac{1}{(\NN\mfp)^{\min\{2,1+\alpha\}}}
	\right).
	\]
	Summing over every such tuple $\eu{P}$ to obtain an upper bound for $\sum_{\eu{P}} |\widehat{\kappa}_{\mc{B},X}(\eu{P})|$, it suffices to show
	\[
	\sum_{\substack{\mfp\notin S\\ \NN\mfp\le X^\psi}}
	\frac{1}{\NN\mfp}\frac{\log\NN\mfp}{\log X}
	+
	\sum_{\substack{\mfp\notin S\\ \NN\mfp\le X^\psi}}
	\frac{1}{(\NN\mfp)^{\min\{2,1+\alpha\}}}
	\ll 1.
	\]
	The second sum is $O(\zeta_k(\min\{2,1+\alpha\})) = O(1)$ where $\zeta_{k}(s)$ is the Dedekind zeta function for $k$, and the first sum is also $O(1)$ using summation by parts and the prime ideal theorem for $k$.
\end{proof}

\begin{lemma}[Higher cumulants are small]
	\label{lem:higher cumulants of the smoothed sequence}
	Let $m \in \Z_{\ge 2}$ and $X \ge X_0$.
	Then for all $0 < \psi \le \min\{1,\frac{1}{m(A_m+1)} (\log \log X)^{-1/2+\delta_m}\}$, one has that
	\[
	\kappa_{\msF,H,X}^{(m)}(\omega_{\psi,X})\ll_m \log\log X.
	\]
\end{lemma}

\begin{proof}
	Applying \cref{lem:approximation of smoothed cumulants,lem:off diagonal tuples are negligible}, we deduce that
	\begin{align*}
	\kappa_{\msF,H,X}^{(m)}(\omega_{\psi,X})
	&=
	\sum_{\substack{\mfp\notin S\\ \NN\mfp\le X^\psi}}
	\omega(\mfp)^m\widehat{\kappa}_{\mc{B},X}(\underbrace{\mfp,\ldots,\mfp}_{m \text{ times}})
	+O_m(1)
	\\ &\ll_m
	\sum_{\substack{\mfp\notin S\\ \NN\mfp\le X^\psi}}
	|\widehat{\kappa}_{\mc{B},X}(\mfp,\ldots,\mfp)|.
	\end{align*}
	For every prime $\mfp \notin S$, the diagonal algebraic cumulant $\widehat{\kappa}_{\mc{B},X}(\mfp,\ldots,\mfp)$ is equal to the $m$-th order cumulant of a (virtual) Bernoulli random variable with mean $\widehat{M}_{\mc{B},X}(\mfp)$, hence $\widehat{\kappa}_{\mc{B},X}(\mfp,\ldots,\mfp) = O_m(\widehat{M}_{\mc{B},X}(\mfp))$.
	By \cref{def:well-mixed Bernoulli process}(i) for $\mc{B}$, one has that
	\[
	\widehat{M}_{\mc{B},X}(\mfp) = 
	\sum_{\lambda \in \Lambda} w_{\lambda}(X) f_{\mfp}^{\lambda}(s_0)
	\ll 1/\NN\mfp.
	\]
	Using the above bound and the fact that $\omega$ is bounded, we have that
	\[
	\kappa_{\msF,H,X}^{(m)}(\omega_{\psi,X})
	\ll_m
	\sum_{\substack{\mfp\notin S\\ \NN\mfp\le X^\psi}}\frac{1}{\NN\mfp}
	\ll \log\log X.
	\]
	This finishes the proof.
\end{proof}

Our goal is to prove a central limit theorem for the distribution associated to the indexed set
\[
\left\{
\frac{\omega(\mfn) - \mathbb{E}_{\msF,H,X}[\omega]}{\sqrt{\log\log X}}
: \mfn \in \msF_H(X)
\right\}, \qquad 
\mathbb{E}_{\msF,H,X}[\omega] := 
\sum_{\mfp} \omega(\mfp) M_{\msF,H,X}(\mfp)
\]
as $X \to \infty$, where $\mfp$ varies over all primes of $k$.
Currently, \cref{lem:higher cumulants of the smoothed sequence} implies for a suitable choice of $\psi = \psi(X)$ that the indexed set
\[
\left\{
\frac{\omega_{\psi,X}(\mfn) - \mathbb{E}_{\msF,H,X}[\omega_{\psi,X}]}{\sqrt{\log\log X}}
: \mfn \in \msF_H(X)
\right\}, \qquad 
\mathbb{E}_{\msF,H,X}[\omega_{\psi,X}] := 
\sum_{\substack{\mfp \notin S \\ \NN\mfp \le X^{\psi}}} \omega(\mfp) M_{\msF,H,X}(\mfp)
\]
weakly converges as $X\to \infty$ to a normal distribution (possibly with variance $0$).
Therefore, it remains to remove the $\psi$-smoothing from this result.

\begin{lemma}[Comparing $\psi$-smoothing]
	\label{lem:removing the smoothing}
	For all $X \ge 10$ and $0 < \psi \le 1$, one has that
	\[
	\sup_{\mfn \in \msF_H(X)}
	\left|
		\frac{\omega(\mfn) - \mathbb{E}_{\msF,H,X}[\omega]}{\sqrt{\log\log X}}
		- \frac{\omega_{\psi,X}(\mfn) - \mathbb{E}_{\msF,H,X}[\omega_{\psi,X}]}{\sqrt{\log\log X}}
	\right|
	\ll
	\frac{1}{\psi \sqrt{\log\log X}}.
	\]
\end{lemma}

\begin{proof}
	Let $\mfn \in \msF_H(X)$.
	The functions $\omega(\mfn)$ and $\omega_{\psi,X}(\mfn)$ differ only at primes in $S$ and at primes dividing $\mfn$ with norm exceeding $X^\psi$.
	The primes in $S$ contribute an error at most $|S|\cdot\Vert \omega \Vert_{\rm sup} = O(1)$, and the primes with norm exceeding $X^{\psi}$ are similarly handled by \cref{def:reasonable sequence for ideals}(i) for $(\msF,H)$.
	This implies the pointwise bound
	\[
	\sup_{\mfn \in \msF_H(X)}
	|\omega(\mfn) - \omega_{\psi,X}(\mfn)| \ll \frac{1}{\psi},
	\]
	therefore one has the average bound $|\mathbb{E}_{\msF,H,X}[\omega] - \mathbb{E}_{\msF,H,X}[\omega_{\psi,X}]| \ll \frac{1}{\psi}$.
	Using these two estimates and dividing by $\sqrt{\log\log X}$ yields the desired inequality.
\end{proof}

\begin{lemma}[$\psi$-smoothed mean and variance formula]
	\label{lem:smoothed mean and variance formula}
	Let $X \ge X_0$.
	For all $0 < \psi \le 1$, one has that
	\[
	\mathbb{E}_{\msF,H,X}[\omega_{\psi,X}] = 
	\sum_{\lambda \in \Lambda^+} w_{\lambda}(\infty)
	\sum_{\substack{\mfp \notin S \\ \NN\mfp \le X^{\psi}}}
	\omega(\mfp)
	f_{\mfp}^{\lambda}(s_0) +
	O\left(
		1
		\right),
	\]
	\[
	\mathbb{E}_{\msF,H,X}[(\omega_{\psi,X})^2] - (\mathbb{E}_{\msF,H,X}[\omega_{\psi,X}])^2 = 
	\sum_{\lambda \in \Lambda^+} w_{\lambda}(\infty)
	\sum_{\substack{\mfp \notin S \\ \NN\mfp \le X^{\psi}}}
	\omega(\mfp)^2
	f_{\mfp}^{\lambda}(s_0) +
	O\left(
		1
		\right).
	\]
\end{lemma}

\begin{proof}
	The expectation $\mathbb{E}_{\msF,H,X}[\bullet]$ is the same as the first cumulant $\kappa_{\msF,H,X}^{(1)}(\bullet)$, and similarly the variance $\mathbb{E}_{\msF,H,X}[\bullet^2] - (\mathbb{E}_{\msF,H,X}[\bullet])^2$ is the same as the second cumulant $\kappa_{\msF,H,X}^{(2)}(\bullet)$.
	We apply \cref{lem:approximation of smoothed cumulants,lem:off diagonal tuples are negligible} to deduce for all $m \in \{1,2\}$ that
	\[
	\kappa_{\msF,H,X}^{(m)}(\omega_{\psi,X}) = 
	\sum_{\substack{\mfp \notin S \\ \NN\mfp \le X^{\psi}}} \omega(\mfp)^m \widehat{\kappa}_{\mc{B},X}(\underbrace{\mfp,\ldots,\mfp}_{m \text{ times}}) + O(1).
	\]
	The diagonal algebraic cumulant $\widehat{\kappa}_{\mc{B},X}(\mfp,\ldots,\mfp)$ is equal to the $m$-th order cumulant of a (virtual) Bernoulli random variable with mean $\widehat{M}_{\mc{B},X}(\mfp)$, hence $\widehat{\kappa}_{\mc{B},X}(\mfp,\ldots,\mfp)$
	\begin{align*}
	 &= 
	\widehat{M}_{\mc{B},X}(\mfp) + O(\widehat{M}_{\mc{B},X}(\mfp)^2) 
	\\ &=
	\widehat{M}_{\mc{B},X}(\mfp) + O\left(\frac{1}{(\NN\mfp)^2}\right)
	\qquad \text{ by \cref{def:well-mixed Bernoulli process}(i)}
	\\ &=
	\sum_{\lambda \in \Lambda} w_{\lambda}(X) f_{\mfp}^{\lambda}(s_0) +
	O\left(
		\frac{1}{\NN\mfp} \frac{\log \NN\mfp}{\log X} + 
		\frac{1}{(\NN\mfp)^2}\right)
		\qquad \text{ as in the proof of \cref{lem:mixed cumulant bound}}
	\\ &=
	\sum_{\lambda \in \Lambda^+} w_{\lambda}(\infty) f_{\mfp}^{\lambda}(s_0) +
	O\left(
		\frac{1}{\log X} \frac{1}{\NN\mfp} + 
		\frac{1}{\NN\mfp}\frac{\log \NN\mfp}{\log X} + 
		\frac{1}{(\NN\mfp)^2}\right)
	\ \text{ by \cref{rem:residue is a polynomial in X} and \cref{def:well-mixed Bernoulli process}(i)}.
	\end{align*}
	By summing over all primes $\mfp \notin S$ with norm $\le X^{\psi}$, this completes the proof.
\end{proof}

\subsection{Proof of \cref{thm:EK criterion for ideals}}
\label{subsec:proof of clt}

For all $X \ge X_0 \ge 10$, we put
\[
\psi(X) := 
\exp\left(
	-\frac{1}{2} \log\log \log 10X + \log\log\log\log 10X
\right) \in (0,1].
\]
The quantities $\mu_X$ and $V_X$ appearing in \cref{thm:EK criterion for ideals} are expressible in our current notation as
\[
\mu_X = \frac{\mathbb{E}_{\msF,H,X}[\omega]}{\log \log X},
\]
\begin{align*}
V_X &= \frac{\mathbb{E}_{\msF,H,X}[\omega^2]-(\mathbb{E}_{\msF,H,X}[\omega])^2
- \sum_{\mfp_1 \neq \mfp_2} \omega(\mfp_1)\omega(\mfp_2) (M_{\msF,H,X}(\mfp_1,\mfp_2) - M_{\msF,H,X}(\mfp_1)M_{\msF,H,X}(\mfp_2))
}{\log\log X}
\end{align*}
It follows from \cref{lem:approximation of smoothed cumulants,lem:off diagonal tuples are negligible,lem:removing the smoothing} that
\begin{equation}
\label{eq:expression for Vx in terms of total variance plus error}
V_X = \frac{\mathbb{E}_{\msF,H,X}[\omega^2]-(\mathbb{E}_{\msF,H,X}[\omega])^2}{\log\log X} + O\left(\frac{1}{\log\log X}\right)
\end{equation}

\begin{proof}[Proof of \cref{thm:EK criterion for ideals}(i)]
	Let $X \ge X_0$.
	We are given that $V_X$ converges to $V>0$ as $X\to \infty$.
	Using Slutsky's theorem to replace $\sqrt{V_X\log\log X}$ with a constant $\sqrt{V} > 0$ times $\sqrt{\log \log X}$, it suffices to prove a central limit theorem for the distribution associated to the indexed set
	\begin{equation}
	\label{eq:indexed set of omega psi X normalized by E and log log X}
		\left\{
		\frac{\omega_{\psi,X}(\mfn) - \mathbb{E}[\omega_{\psi,X}]}{\sqrt{\log\log X}}
		: \mfn \in \msF_H(X)
	\right\}
	\end{equation}
	as $X\to \infty$.
	For this, we observe by \cref{lem:higher cumulants of the smoothed sequence} that for every $m\in\Z_{\ge 3}$ and $X$ sufficiently large in terms of $m$ so that
	\[
	\exp(-\delta_m \log\log\log10X + \log\log\log\log10X) < \frac{1}{m(A_m+1)},
	\]
	the $m$-th cumulant of (the distribution associated to the indexed set) \cref{eq:indexed set of omega psi X normalized by E and log log X} is given by
	\[
	\frac{\kappa_{\msF,H,X}^{(m)}(\omega_{\psi,X})}{(\log\log X)^{m/2}} \ll_m
	(\log\log X)^{1-m/2} = o_m(1).
	\]
	Thus, the $m$-th cumulant of \cref{eq:indexed set of omega psi X normalized by E and log log X} tends to zero as $X \to \infty$.
	The first cumulant (i.e.\ mean) of \cref{eq:indexed set of omega psi X normalized by E and log log X} is identically $0$, and \cref{lem:removing the smoothing,eq:expression for Vx in terms of total variance plus error} imply that the second cumulant (i.e.\ variance) of \cref{eq:indexed set of omega psi X normalized by E and log log X} converges to $V$ as $X\to \infty$.
	Thus, the limiting sequence of ($\N$-)cumulants of \cref{eq:indexed set of omega psi X normalized by E and log log X} as $X\to \infty$ is given by
	\begin{equation}
		\label{eq:sequence of cumulants in the limit}
		(0,V,0,0,0,\ldots).
	\end{equation}

	The moments of a real random variable are expressed as universal polynomials in terms of the cumulants; therefore, the sequence of cumulants uniquely determines the moments.
	It is well-known \cite[Ex.\ 30.1]{BillingsleyBook3rdEd} that the normal distribution with mean $0$ and variance $V>0$ is the unique distribution on $\R$ whose sequence of cumulants is \cref{eq:sequence of cumulants in the limit}; therefore, by the method of moments for real random variables \cite[Thm.\ 30.2]{BillingsleyBook3rdEd} we deduce that the distribution associated to \cref{eq:indexed set of omega psi X normalized by E and log log X} weakly converges to a normal distribution as $X\to \infty$.
	\cref{thm:EK criterion for ideals}(i) then follows from the Portmanteau theorem.
\end{proof}

\begin{proof}[Proof of \cref{thm:EK criterion for ideals}(ii)]
	Let $X \ge X_0$.
	By \cref{eq:expression for Vx in terms of total variance plus error}, $V = \lim\limits_{X \to \infty} V_X$ exists if and only if
	\begin{align*}
	V &= \lim_{X \to \infty} \frac{\mathbb{E}_{\msF,H,X}[(\omega_{\psi,X})^2] - (\mathbb{E}_{\msF,H,X}[\omega_{\psi,X}])^2}{\log\log X} \\
	&=
	\sum_{\lambda \in \Lambda^+} w_{\lambda}(\infty)
	\lim_{X \to \infty} \frac{1}{\log\log X} \sum_{\substack{\mfp \notin S \\ \NN\mfp \le X^{\psi}}} \omega(\mfp)^2 f_{\mfp}^{\lambda}(s_0)
	\qquad \text{by \cref{lem:smoothed mean and variance formula}} \\ &=
	\sum_{\lambda \in \Lambda^+} w_{\lambda}(\infty)
	\lim_{X \to \infty} \frac{1}{\log\log X} \sum_{\substack{\mfp \notin S \\ \NN\mfp \le X}} \omega(\mfp)^2 f_{\mfp}^{\lambda}(s_0),
	\end{align*}
	and a similar argument holds for showing that $\mu = \lim\limits_{X\to \infty} \mu_X$ exists if and only if
	\[
	\mu = 
	\sum_{\lambda \in \Lambda^+} w_{\lambda}(\infty)
	\lim_{X \to \infty} \frac{1}{\log\log X} \sum_{\substack{\mfp \notin S \\ \NN\mfp \le X}} \omega(\mfp) f_{\mfp}^{\lambda}(s_0).
	\]
	This finishes the proof.
\end{proof}

\begin{proof}[Proof of \cref{thm:EK criterion for ideals}(iii)]
	Let $X \ge X_0$.
	We are given that $\mu_X$ and $V_X$ converge to $\mu \in \R$ and $V > 0$ respectively as $X\to \infty$, and also that
	\[
	\sum_{\lambda \in \Lambda^+} w_{\lambda}(\infty)
	\sum_{\substack{\mfp \notin S \\ \NN\mfp \le X}} \omega(\mfp) f_\mfp^\lambda(s_0) = 
	\mu \log\log X + o(\sqrt{\log\log X}).
	\]
	We have that $\mu_X \log\log X$ is equal to
	\begin{align*}
			\mathbb{E}_{\msF,H,X}[\omega] &=
				\mathbb{E}_{\msF,H,X}[\omega_{\psi,X}] + O\left(\frac{1}{\psi}\right) \qquad \text{ by \cref{def:reasonable sequence for ideals}(i)}
			\\ &=
				\sum_{\lambda \in \Lambda^+} w_{\lambda}(\infty)
				\sum_{\substack{\mfp \notin S \\ \NN\mfp \le X^{\psi}}} \omega(\mfp) f_{\mfp}^{\lambda}(s_0)
				+ O\left(1+\frac{1}{\psi}\right) \quad \text{ by \cref{lem:smoothed mean and variance formula}}
			\\ &=
			\mu \log\log X + o(\sqrt{\log\log X}) + O\left(1+\frac{1}{\psi}\right).
	\end{align*}
	This shows that
	\[
	\sup_{\mfn \in \msF_H(X)}
	\left|
		\frac{\omega_{\psi,X}(\mfn) - \mu_X \log\log X}{\sqrt{\log\log X}} -
		\frac{\omega_{\psi,X}(\mfn) - \mu \log\log X}{\sqrt{\log\log X}}
	\right| \ll 
	|o(1)| + 
	\frac{1}{\psi \sqrt{\log\log X}} \ll |o(1)|,
	\]
	which finishes the proof.
\end{proof}

\nocite{NeukirchANTbook}
\nocite{ElBazLoughranSofos}
\bibliography{bib.bib}

@article {LemkeOliverThorneErdosKac,
	AUTHOR = {Lemke Oliver, Robert J. and Thorne, Frank},
	TITLE = {The number of ramified primes in number fields of small
	degree},
	JOURNAL = {Proc. Amer. Math. Soc.},
	FJOURNAL = {Proceedings of the American Mathematical Society},
	VOLUME = {145},
	YEAR = {2017},
	NUMBER = {8},
	PAGES = {3201--3210},
	ISSN = {0002-9939,1088-6826},
	MRCLASS = {11R21 (11K65 11R16)},
	MRNUMBER = {3652776},
	MRREVIEWER = {Jerry\ Hu},
	DOI = {10.1090/proc/13467},
	URL = {https://doi.org/10.1090/proc/13467},
}

@article {ErdosKacOriginalPaper,
	AUTHOR = {Erd\H{o}s, Paul and Kac, Mark},
	TITLE = {The {G}aussian law of errors in the theory of additive number
	theoretic functions},
	JOURNAL = {Amer. J. Math.},
	FJOURNAL = {American Journal of Mathematics},
	VOLUME = {62},
	YEAR = {1940},
	PAGES = {738--742},
	ISSN = {0002-9327,1080-6377},
	MRCLASS = {10.0X},
	MRNUMBER = {2374},
	MRREVIEWER = {E.\ R.\ van Kampen},
	DOI = {10.2307/2371483},
	URL = {https://doi.org/10.2307/2371483},
}

@misc{LoughranSantensMalleBrauer,
	title={Malle's conjecture and {B}rauer groups of stacks}, 
	author={Daniel Loughran and Tim Santens},
	year={2024},
	eprint={2412.04196},
	archivePrefix={arXiv},
	primaryClass={math.NT},
	url={https://arxiv.org/abs/2412.04196}, 
	note = {\url{https://arxiv.org/abs/2412.04196}},
}

@article {MalleMCI,
	AUTHOR = {Malle, Gunter},
	TITLE = {On the distribution of {G}alois groups},
	JOURNAL = {J. Number Theory},
	FJOURNAL = {Journal of Number Theory},
	VOLUME = {92},
	YEAR = {2002},
	NUMBER = {2},
	PAGES = {315--329},
	ISSN = {0022-314X,1096-1658},
	MRCLASS = {12F12 (11R32 11R47 12F10)},
	MRNUMBER = {1884706},
	MRREVIEWER = {Helmut\ V\"olklein},
	DOI = {10.1006/jnth.2001.2713},
	URL = {https://doi.org/10.1006/jnth.2001.2713},
}

@article {MalleMCII,
	AUTHOR = {Malle, Gunter},
	TITLE = {On the distribution of {G}alois groups. {II}},
	JOURNAL = {Experiment. Math.},
	FJOURNAL = {Experimental Mathematics},
	VOLUME = {13},
	YEAR = {2004},
	NUMBER = {2},
	PAGES = {129--135},
	ISSN = {1058-6458,1944-950X},
	MRCLASS = {11R32 (11R47 12F10)},
	MRNUMBER = {2068887},
	MRREVIEWER = {F.\ Diaz y Diaz},
	URL = {http://projecteuclid.org/euclid.em/1090350928},
}

@article {NoteAdditiveFunctions,
	AUTHOR = {Delange, Hubert and Halberstam, Heini},
	TITLE = {A note on additive functions},
	JOURNAL = {Pacific J. Math.},
	FJOURNAL = {Pacific Journal of Mathematics},
	VOLUME = {7},
	YEAR = {1957},
	PAGES = {1551--1556},
	ISSN = {0030-8730,1945-5844},
	MRCLASS = {10.1X},
	MRNUMBER = {92820},
	MRREVIEWER = {H.\ N.\ Shapiro},
	URL = {http://projecteuclid.org/euclid.pjm/1103043227},
}

@article {BillingsleyEK,
	AUTHOR = {Billingsley, Patrick},
	TITLE = {On the central limit theorem for the prime divisor functions},
	JOURNAL = {Amer. Math. Monthly},
	FJOURNAL = {American Mathematical Monthly},
	VOLUME = {76},
	YEAR = {1969},
	PAGES = {132--139},
	ISSN = {0002-9890,1930-0972},
	MRCLASS = {60.30},
	MRNUMBER = {242222},
	MRREVIEWER = {J.\ Sethuraman},
	DOI = {10.2307/2317259},
	URL = {https://doi.org/10.2307/2317259},
}

@article {WoodLocalAbelianProbabilities,
	AUTHOR = {Wood, Melanie Matchett},
	TITLE = {On the probabilities of local behaviors in abelian field
	extensions},
	JOURNAL = {Compos. Math.},
	FJOURNAL = {Compositio Mathematica},
	VOLUME = {146},
	YEAR = {2010},
	NUMBER = {1},
	PAGES = {102--128},
	ISSN = {0010-437X,1570-5846},
	MRCLASS = {11R20 (11R45)},
	MRNUMBER = {2581243},
	MRREVIEWER = {Roland\ Qu\^eme},
	DOI = {10.1112/S0010437X0900431X},
	URL = {https://doi.org/10.1112/S0010437X0900431X},
	NOTE = {See corrigendum.},
}

@misc{AlbertsPowerSavingsAbelianExtensions,
	title={Power Savings for Counting (Twisted) Abelian Extensions of Number Fields}, 
	author={Brandon Alberts},
	year={2024},
	eprint={2402.03475},
	archivePrefix={arXiv},
	primaryClass={math.NT},
	url={https://arxiv.org/abs/2402.03475}, 
	note={\url{https://arxiv.org/abs/2402.03475}}
}

@article {FreiLoughranNewtonPowerSavingsAbelianExtensions,
	AUTHOR = {Frei, Christopher and Loughran, Daniel and Newton, Rachel},
	TITLE = {The {H}asse norm principle for abelian extensions},
	JOURNAL = {Amer. J. Math.},
	FJOURNAL = {American Journal of Mathematics},
	VOLUME = {140},
	YEAR = {2018},
	NUMBER = {6},
	PAGES = {1639--1685},
	ISSN = {0002-9327,1080-6377},
	MRCLASS = {11R37 (11R45)},
	MRNUMBER = {3884640},
	MRREVIEWER = {Claudio\ Stirpe},
	DOI = {10.1353/ajm.2018.0048},
	URL = {https://doi.org/10.1353/ajm.2018.0048},
}

@article {WrightAbelianExtensions,
	AUTHOR = {Wright, David J.},
	TITLE = {Distribution of discriminants of abelian extensions},
	JOURNAL = {Proc. London Math. Soc. (3)},
	FJOURNAL = {Proceedings of the London Mathematical Society. Third Series},
	VOLUME = {58},
	YEAR = {1989},
	NUMBER = {1},
	PAGES = {17--50},
	ISSN = {0024-6115,1460-244X},
	MRCLASS = {11R29 (11R37 11R45)},
	MRNUMBER = {969545},
	MRREVIEWER = {Boris\ Datskovsky},
	DOI = {10.1112/plms/s3-58.1.17},
	URL = {https://doi.org/10.1112/plms/s3-58.1.17},
}

@misc{AlbertsAveragedInputTauberianTheorems,
	title={An Explicit {T}auberian Theorem taking Averaged Inputs with an Application to Counting Abelian Number Fields}, 
	author={Brandon Alberts},
	year={2025},
	eprint={2508.20814},
	archivePrefix={arXiv},
	primaryClass={math.NT},
	url={https://arxiv.org/abs/2508.20814}, 
	note = {\url{https://arxiv.org/abs/2508.20814}}
}

@misc{TavernierRestrictedRamificationAbelianCount,
	title={Counting abelian number fields with restricted ramification type}, 
	author={Julie Tavernier},
	year={2025},
	eprint={2507.00448},
	archivePrefix={arXiv},
	primaryClass={math.NT},
	url={https://arxiv.org/abs/2507.00448}, 
	note={\url{https://arxiv.org/abs/2507.00448}}
}

@article {AlbertsODorney,
    AUTHOR = {Alberts, Brandon and O'Dorney, Evan},
     TITLE = {Harmonic analysis and statistics of the first {G}alois
              cohomology group},
   JOURNAL = {Res. Math. Sci.},
  FJOURNAL = {Research in the Mathematical Sciences},
    VOLUME = {8},
      YEAR = {2021},
    NUMBER = {3},
     PAGES = {Paper No. 50, 16},
      ISSN = {2522-0144,2197-9847},
   MRCLASS = {11N45 (11M45 11R34 43A25)},
  MRNUMBER = {4298097},
MRREVIEWER = {Claudio\ Stirpe},
       DOI = {10.1007/s40687-021-00283-2},
       URL = {https://doi.org/10.1007/s40687-021-00283-2},
}

@book {BillingsleyBook3rdEd,
    AUTHOR = {Billingsley, Patrick},
     TITLE = {Probability and measure},
    SERIES = {Wiley Series in Probability and Mathematical Statistics},
   EDITION = {Third},
      NOTE = {A Wiley-Interscience Publication},
 PUBLISHER = {John Wiley \& Sons, Inc., New York},
      YEAR = {1995},
     PAGES = {xiv+593},
      ISBN = {0-471-00710-2},
   MRCLASS = {60-01 (28-01)},
  MRNUMBER = {1324786},
}

@article {AlbertsTwistedMalle,
    AUTHOR = {Alberts, Brandon},
     TITLE = {Statistics of the first {G}alois cohomology group: a
              refinement of {M}alle's conjecture},
   JOURNAL = {Algebra Number Theory},
  FJOURNAL = {Algebra \& Number Theory},
    VOLUME = {15},
      YEAR = {2021},
    NUMBER = {10},
     PAGES = {2513--2569},
      ISSN = {1937-0652,1944-7833},
   MRCLASS = {11N45 (11R21 11R32 11R34)},
  MRNUMBER = {4377858},
MRREVIEWER = {Peter\ Koymans},
       DOI = {10.2140/ant.2021.15.2513},
       URL = {https://doi.org/10.2140/ant.2021.15.2513},
}

@book {NeukirchANTbook,
    AUTHOR = {Neukirch, J\"urgen},
     TITLE = {Algebraic number theory},
    SERIES = {Grundlehren der mathematischen Wissenschaften [Fundamental
              Principles of Mathematical Sciences]},
    VOLUME = {322},
      NOTE = {Translated from the 1992 German original and with a note by
              Norbert Schappacher,
              With a foreword by G. Harder},
 PUBLISHER = {Springer-Verlag, Berlin},
      YEAR = {1999},
     PAGES = {xviii+571},
      ISBN = {3-540-65399-6},
   MRCLASS = {11Rxx (11-02 11S15 11S31 14C40)},
  MRNUMBER = {1697859},
MRREVIEWER = {Cornelius\ Greither},
       DOI = {10.1007/978-3-662-03983-0},
       URL = {https://doi.org/10.1007/978-3-662-03983-0},
}

@article {ElBazLoughranSofos,
    AUTHOR = {El-Baz, Daniel and Loughran, Daniel and Sofos, Efthymios},
     TITLE = {Multivariate normal distribution for integral points on
              varieties},
   JOURNAL = {Trans. Amer. Math. Soc.},
  FJOURNAL = {Transactions of the American Mathematical Society},
    VOLUME = {375},
      YEAR = {2022},
    NUMBER = {5},
     PAGES = {3089--3128},
      ISSN = {0002-9947,1088-6850},
   MRCLASS = {11N36 (14G05 60F05)},
  MRNUMBER = {4402657},
MRREVIEWER = {Donald\ Jason\ Gibson},
       DOI = {10.1090/tran/8545},
       URL = {https://doi.org/10.1090/tran/8545},
}

@article {DardaYasudaTorsorsFiniteGroupSchemes,
    AUTHOR = {Darda, Ratko and Yasuda, Takehiko},
     TITLE = {Torsors for finite group schemes of bounded height},
   JOURNAL = {J. Lond. Math. Soc. (2)},
  FJOURNAL = {Journal of the London Mathematical Society. Second Series},
    VOLUME = {108},
      YEAR = {2023},
    NUMBER = {3},
     PAGES = {1275--1331},
      ISSN = {0024-6107,1469-7750},
   MRCLASS = {11G50 (11R32 11R34 11R45 11R56 14L15)},
  MRNUMBER = {4639951},
MRREVIEWER = {Christopher\ Frei},
       DOI = {10.1112/jlms.12780},
       URL = {https://doi.org/10.1112/jlms.12780},
}

@book {NicaSpeicherLecturesOnFreeProbability,
    AUTHOR = {Nica, Alexandru and Speicher, Roland},
     TITLE = {Lectures on the combinatorics of free probability},
    SERIES = {London Mathematical Society Lecture Note Series},
    VOLUME = {335},
 PUBLISHER = {Cambridge University Press, Cambridge},
      YEAR = {2006},
     PAGES = {xvi+417},
      ISBN = {978-0-521-85852-6; 0-521-85852-6},
   MRCLASS = {46L54 (46L53 60C05 81S25)},
  MRNUMBER = {2266879},
MRREVIEWER = {Todd\ Kemp},
       DOI = {10.1017/CBO9780511735127},
       URL = {https://doi.org/10.1017/CBO9780511735127},
}

@misc{LoughranSantensSurvey,
      title={The leading constant in {M}alle's conjecture}, 
      author={Daniel Loughran and Tim Santens},
      year={2026},
      eprint={2606.04983},
      archivePrefix={arXiv},
      primaryClass={math.NT},
      url={https://arxiv.org/abs/2606.04983}, 
      note = {\url{https://arxiv.org/abs/2606.04983}}
}
\bibliographystyle{alpha-fullkey}

\end{document}